\documentclass[reqno,a4paper,10pt]{amsart}

\usepackage[mathscr]{eucal}
\usepackage{url} 
\usepackage{mathabx}
\usepackage{amssymb,amsfonts,amsmath,latexsym,bm,mathtools,comment}
\usepackage{multirow}

\usepackage{a4wide}

\newtheorem{thm}{Theorem}[section]
\newtheorem{lem}[thm]{Lemma}
\newtheorem{coro}[thm]{Corollary}
\newtheorem{prop}[thm]{Proposition}

\theoremstyle{definition}

\newtheorem{ex}[thm]{Example}

\newenvironment{newlist}
   {\begin{list}{}{\setlength{\labelsep}{0.25cm}
                   \setlength{\labelwidth}{0.65cm}
                      \setlength{\leftmargin}{0.9cm}}}
   {\end{list}}

\newcommand{\class}[1]{\boldsymbol{\mathcal{#1}}}
\newcommand{\cat}[1]{\boldsymbol{\mathscr{#1}}}
\newcommand{\alg}[1]{\mathbf{#1}}
\newcommand{\fnt}[1]{\mathsf{#1}}
\newcommand{\ope}[1]{\mathbb{#1}}

\newcommand{\duplicateV}{\class A}
\newcommand{\baseV}{\class B}
\newcommand{\CM}{\class M}
\newcommand{\CN}{\class N}
\newcommand{\unbounded}[1]{#1\mbox{\tiny{$u$}}}
\newcommand{\CL}{{\cat L}}   
\newcommand{\CLU}{\unbounded{\CL}}  
\newcommand{\CB}{\cat B}
\newcommand{\CGB}{\cat{GB}} 
\newcommand{\DM}{\cat{DM}} 
\newcommand{\KL}{\cat{K}} 
\newcommand{\CCD}{\cat D}
\newcommand{\CCDU}{\unbounded{\CCD}}
\newcommand{\TL}{{\cat{DT}}}  
\newcommand{\IT}{\cat{IT}}
\newcommand{\DBC}{\cat{DBC}}
\newcommand{\DB}{\cat{DB}}
\newcommand{\BL}{\cat{BL}}
\newcommand{\DBCU}{\unbounded{\DBC}}
\newcommand{\DBU}{\unbounded{\DB}}
\newcommand{\BLU}{\unbounded{\BL}}
\newcommand{\DMU}{\unbounded{\DM}}
\newcommand{\krB}{\BL_{\to_k}}
\newcommand{\trB}{\BL_{\to_t}}

\newcommand{\Sigmavar}[1]{\Sigma_{\scriptscriptstyle #1}}
\newcommand{\Gammavari}[1]{\Gamma_{\scriptscriptstyle #1}}

\newcommand{\A}{\alg{A}}
\newcommand{\B}{\alg{B}}
\newcommand{\C}{\alg{C}}
\newcommand{\M}{\alg{M}}
\newcommand{\N}{\alg{N}}
\newcommand{\PP}{\alg{P}}  
\newcommand{\Q}{\alg{Q}}
\newcommand{\two}{\boldsymbol 2}
\newcommand{\Lalg}{\alg{L}}
\newcommand{\three}{\boldsymbol 3}
\newcommand{\four}{\boldsymbol 4}
\newcommand{\Nine}{\boldsymbol 9}
\newcommand{\NineS}{\Ninevar {/}}
\newcommand{\twovar}[1]{\two_ {\scriptscriptstyle #1}}
\newcommand{\threevar}[1]{\three_ {\scriptscriptstyle #1}}
\newcommand{\fourvar}[1]{\four_ {\scriptscriptstyle #1}}
\newcommand{\Ninevar}[1]{\Nine_ {\scriptscriptstyle #1}}
\newcommand{\Svar}[1]{{\bf 16}_ {\scriptscriptstyle #1}}

\DeclareMathOperator{\ISP}{\ope{ISP}}
 \DeclareMathOperator{\HSP}{\ope{HSP}}
\DeclareMathOperator{\aaimp}{\supset}
\newcommand{\fitimp}{\mathrel{:}} 
\newcommand{\ldiv}{\mathrel{\backslash}} 
\newcommand{\rdiv}{\mathrel{/}} 
 \DeclareMathOperator{\OurSlash}{\slash}
  \DeclareMathOperator{\Con}{Con}
\newcommand{\odotG}[1]{\mathrel{{\odot}_{#1}}} 

\newcommand{\dupl}{\llcurly}

\renewcommand{\leq}{\leqslant}
\renewcommand{\geq}{\geqslant}

\begin{document}

\title[A general framework for product representations]{A general framework for 
product representations: bilattices and beyond}

\author{L.M. Cabrer}

\address[L.M. Cabrer]{Dipartimento di Statistica, Informatica, Applicazioni\\
 Universit\`a degli Studi di Firenze\\
59 Viale Morgani, 50134\\
Florence, Italy
}
\email{l.cabrer@disia.unifi.it}

\author{H.A. Priestley}
\address[H.A. Priestley]{
Mathematical Institute\\
University of Oxford\\
 Radcliffe Observatory Quarter,
Oxford, OX2 6GG, United Kingdom}
\email{hap@maths.ox.ac.uk}

\begin{abstract}
This paper studies algebras arising as algebraic semantics 
 for logics used to model reasoning 
with incomplete or inconsistent information.  In particular we study, 
in a uniform way,
varieties of bilattices
equipped with additional logic-related operations and their product representations.   
 Our principal result is 
a
 very general product representation theorem.
 Specifically, we present a syntactic 
procedure (called \emph{duplication}) for building a product algebra out of  a given base algebra
and a given set of terms. The procedure 
lifts functorially to the generated varieties and leads, under specified
sufficient conditions, to a categorical equivalence between these varieties.
When these conditions are satisfied, 
a very tight algebraic relationship exists 
between the  base variety and  the enriched variety.  
Moreover varieties arising as duplicates of a common base variety are
automatically categorically equivalent to each other.  
Two further product representation constructions are also presented; these
are in the same spirit as our main theorem and extend the scope of our analysis.

Our catalogue of applications selects varieties for which 
product representations have  previously been obtained  one by one, or which are new. 
We also reveal that certain varieties arising from the modelling of quite
different operations are categorically equivalent. 
Among the  range of examples presented,
 we draw attention in particular to our systematic
treatment of trilattices.
\end{abstract}

\maketitle

\section{Introduction} \label{sec:intro}

The notion of product representation plays a central role in  the study of interlaced bilattices, 
with and without any or all of bounds, negation and additional operations 
(see \textit{inter alia} \cite{Av96,MPSV,RPhD,BJR11,BR13,JR,D13}).
Such algebraic structures have been identified by researchers in artificial intelligence and in philosophical logic 
as of value for analysing  scenarios in which  information may be incomplete or inconsistent.  
The literature in the area is now very extensive.
Following the introduction of bilattices by Ginsberg~\cite{Gin},  
various associated logical systems  were proposed and studied, 
{\it inter alia} by Belnap~\cite{ND3}, Fitting~\cite{F91,FKlgen,Fnice}, 
Avron and Arieli~\cite{AA1} 
and,  more recently, 
by Rivieccio, alone  and in collaboration with Bou and 
Jansana~\cite{RPhD,BR11,R13};
note also the survey by Gargov~\cite{G99}.    
 Moreover, much research  has been done on  algebraic structures having bilattice reducts 
(for example bilattices with an additional operation such as a modality or an
 implication~\cite{GinMod,AA2,BJR11,BR13,R14})
and  also trilattices \cite{Wi1995,SDT,SW}. 
 A bewildering proliferation of examples 
has resulted, with most of the analysis done on a case-by-case basis. 

Our objective in this paper  is to develop an abstract  framework for product representations.  
Our principal result is  Theorem~\ref{thm:GeneralProductEquiv}.  
Our treatment  scores over 
the traditional one in three ways.
Firstly,  product representation theorems have 
traditionally
 been obtained on  a case-by-case basis,
whereas our theorem applies  in a uniform way to many varieties,  
as we shall see in  Sections~\ref{sec:Conflation}--\ref{Sec:Examples}. 
Secondly, the theorem splits the construction  of
a product representation for a variety $\duplicateV$ into two parts.
First we identify a set $\CM$ of algebras (frequently  a single algebra)  that
generates  
$\duplicateV$. 
We then set up the product representation just for the members of  $\CM$.
Then Theorem~\ref{thm:GeneralProductEquiv} automatically 
proves that each element of  $\duplicateV$ admits a product representation. 
Thirdly, the theorem supplies  a categorical equivalence from the outset;
in the literature 
product representation theorems have often been  given 
only at the object level and,   
where such representations  were upgraded 
to categorical equivalences, considerable effort had to be 
 expended for each individual class.

We now present in a little more detail the idea underlying our approach.
Consider  two classes of algebras:   $\duplicateV$, a variety we wish to analyse,  
and a base variety $\baseV$, which we assume to be of the form
$\baseV=\mathbb{V}(\N)$,  the variety generated by some algebra $\N$.
 (The single algebra $\N$ above could be replaced by a class $\CN$ of algebras of common type.) 
Then, when suitable conditions are satisfied, we can `duplicate' $\N$
 to construct  an algebra $\M:= \fnt{P}_{\Gamma}(\N)$  in~$\duplicateV$.
Here the universe of  $\M$ is $N \times N$, where~$N$ is the universe of~$\N$.
The operations in the product are built from $\Gamma$,
a set of pairs of algebraic terms in the base language (that of~$\baseV$),  
used to define certain operations coordinatewise, and are 
combined with coordinate manipulation to link the factors.  
 The set $\Gamma$ is 
called a \emph{duplicator} (for $\baseV$).
Moreover  the duplication construction 
lifts to a category equivalence between the base variety 
$\baseV=\mathbb{V} (\N)$ and the variety 
$\mathbb{V}(\M)$.  In practice, the latter 
is likely to be
 the variety $\mathbb{V}(\fnt{P}_{\Gamma}(\CN))$ we are interested in. 
The mechanism of duplication 
is rooted in the manipulation of terms in an abstract algebraic language.  
Indeed, from this perspective  product representations can be seen to
arise   just from a glorified form of term-equivalence
 (see the discussion before Theorem~\ref{thm:mPR}).  
We stress that the  construction 
does not depend on the specific algebraic language 
of the base class nor that of the duplicated one but only on the relation between their two languages.
We shall follow  the 
literature on product representations
in confining  our examples to  varieties  of bilattice-based algebras.
However  
 the scope of Theorem~\ref{thm:GeneralProductEquiv} 
is not  restricted  to such classes. 

As we shall demonstrate in Sections~\ref{sec:Conflation}--\ref{Sec:Examples},
distributive lattices, Boolean algebras, Heyting algebras, distributive bilattices, 
and De Morgan algebras will serve as base varieties in this way, 
as do their unbounded analogues. 
 The duplicated varieties carry, besides operations from the base language, 
operations which are order-preserving or order-reversing unary involutions;
implication-like operations;  assorted other
logic-driven unary and binary operations; further pairs of lattice operations. 
We stress that the duplication formalism helps  guide us to the product representations we seek.  
To illustrate the point,   we contrast  our treatment of 
distributive bilattices with conflation  in Section~\ref{sec:Conflation} 
with Fitting's account in~\cite{FKlkids} and note our remarks on implicative bilattices  (Example~\ref{ex:implic}).

The generalised form of  product representation 
 given in Section~\ref{Sec:Conclusion} 
takes its cue from two varieties:  pre-bilattices 
(not covered by Theorem~\ref{thm:GeneralProductEquiv}) and interlaced trilattices
(covered, but only by carrying out a two-stage 
duplication).
In an appendix we bring our multitude of examples  together in two tables.
Table~\ref{table1} lists bilattice-based varieties and the base varieties they duplicate, 
and so highlights the categorical equivalences revealed by our analysis.  
Table~\ref{table2}  systematises  the product representations available for
interlaced  bilattices, for 
interlaced trilattices and for interlaced trilattices augmented with one, two or three 
involutory operations.

This work has grown out of
our study of natural dualities for bilattices and their connection with product representations 
\cite{CCP,CPOne}. In \cite{TwoPlus} we return to the duality theme  and set  
up  an automatic procedure to obtain natural dualities for 
classes of algebras that fit into   the general framework 
for product representations presented in this paper.


\section{Preliminaries on bilattices and product representation}\label{Sec:Preliminaries}

 Our investigations involve classes of algebras.  Accordingly  we shall draw
on some of the basic formalism of universal algebra, specifically regarding
algebras, terms and varieties ({\it alias} equational classes); a standard reference is \cite{BS}; see also 
\cite[Chapter~I]{BD} 
for a categorical perspective.
We write $\mathbb V(\CN)$ to denote the variety generated by  a family
$\CN$ of algebras having a  common language. 
Equivalently $\mathbb V(\CN)$ is the class  
$\HSP(\CN)$  of homomorphic images of subalgebras of products of algebras in $\CN$.
We often encounter
 classes such that $\HSP(\CN)=\ISP(\CN)$,  
the class of isomorphic images of subalgebras of products of algebras in~$\CN$. 
We note the elementary but useful fact that
 an algebra $\A$ belongs to $\ISP(\CN)$ if and only if 
the family of homomorphisms from~$\A$ into the algebras in $\CN$ separates the elements of $\A$.
Most often in our investigations
$\CN$ will contain a single algebra~$\N$. 
When this is the case, to simplify the notation, we write~$\N$ instead of $\{\N\}$.
A class of algebras of common language will be regarded as a category in the usual way: 
we take morphisms all homomorphisms.

The algebras we consider as examples will be lattice-based, that is, they 
have reducts in the variety  $\CLU$ of all lattices, with basic operations 
$\lor$ and $\land$.   Here the subscript $_u$ indicates that the lattices are \emph{unbounded}
 in the sense that bottom  and top elements for the underlying order, even when these exist, 
 are not included in the language. 
 We write $\CL$ for the variety of  bounded lattices, 
 {\it viz.}~algebras
$ (L;\lor,\land,0,1)$,  where $(L;\lor,\land) \in \CLU$, and $0,1$ are respectively,
bottom and top elements for the underlying order on~$L$. 
 For any lattice $\Lalg$, unbounded or bounded,  we write $\Lalg^\partial$
  to denote the lattice on the same underlying set, but with the order and bounds (when 
present) reversed.  

We now turn  to bilattices.
We shall assume that readers are familiar with the basic notions;
summaries can be found, for example, in \cite{RPhD, BJR11}.  
Here we establish notation  and  terminology, and 
make only a few comments to set the scene for our study. 
 An \emph{{\upshape(}unbounded{\upshape)}  pre-bilattice}
$\A = (A; \lor_t,\land_t,\lor_k,\land_k)$ is an algebra for which 
$(A;\lor_t,\land_t)$ and $(A;\lor_k,\land_k)$ belong to $\CLU$.  
Here the subscripts $_t$ and $_k$ have the connotation  
of `truth' and `knowledge' and refer to the associated lattices $\A_t$ and $\A_k$
 as the truth and knowledge lattices of $\A$; the corresponding 
  lattice orders are denoted by $\leq_t$ and $\leq_k$.  
Analogous definitions can be formulated in the bounded case.
Here we follow the notation we used in \cite{CPOne} and choose 
 to deviate from that adopted in recent bilattice literature, 
 in which the truth operations are denoted $\lor$ and $\land$ and the knowledge operations by
$\oplus$ and $\otimes$.     

Here, as in \cite{RPhD,BJR11} and elsewhere, the term \emph{bilattice}
is reserved for an algebra $\A$ which is a pre-bilattice enriched with  
a negation operation $\neg$, which is required  to be
an involution  that preserves $\leq_k$ and reverses 
$\leq_t$. 
We shall normally   
assume that 
a negation operator is present, and  delay until
 Section~\ref{Sec:Conclusion} the adaptation  of our approach to  
encompass also the product representation for pre-bilattices.
Unlike negation, whose inclusion or omission leads to significantly  different 
outcomes, whether or not the algebraic language includes nullary operations 
interpreted as lattice bounds is largely a matter of choice,  governed 
for example by the logic being modelled.  Thus we are ambivalent about
constants, sometimes including them and sometimes not; the 
adaptations required for the other case are generally minor.

 An interaction between the lattice operations $\lor_t$, $\land_t$ and  $\lor_k$,
$\land_k$  of a bilattice is needed for a good structure theory.  At a minimum, we need
to impose the condition of \emph{interlacing}, asserting that the operations
in $\{\lor_t, \land_t\}$  and in $\{\lor_k, \land_k\}$ are monotonic with respect to 
$\leq_k$ and $\leq_t$, respectively.  Interlacing is both necessary and sufficient 
for the existence of a product representation (see \cite{RPhD} and also \cite{D13}). 
 We write $\BLU$ and $\BL$ for the varieties of unbounded and bounded interlaced bilattices, respectively.
We recall the product representation for interlaced unbounded bilattices. Given a lattice
 $\Lalg  =(L; \lor,\land)$, then 
$\Lalg \odot \Lalg $ denotes the bilattice with  universe 
$L \times L$ and lattice operations given by
\begin{alignat*}{2}
(a_1,a_2) \lor_t (b_1,b_2)  & = (a_1 \lor b_1, a_2 \land b_2), \qquad & 
(a_1,a_2) \lor_k (b_1,b_2)  & = (a_1 \lor b_1, a_2 \lor b_2), \\
(a_1,a_2) \land_t (b_1,b_2)  & = (a_1 \land b_1, a_2 \lor b_2), \qquad &
(a_1,a_2) \land _k (b_1,b_2)  & = (a_1 \land b_1, a_2 \land b_2);
\end{alignat*}
negation is given by $\neg(a,b) = (b,a)$.
The Product Representation Theorem for unbounded interlaced bilattices
states that,
given $\A \in \BLU$,   
there exists $\Lalg= (L; \lor,\land) \in \CLU$ such that $\A \cong \Lalg \odot \Lalg$.

We can see that  the operations of $\Lalg \odot \Lalg $ are constructed from the operations of
 $\Lalg$ just by manipulating coordinates and applying to them the operations in $\Lalg$. 
 This simple observation is the starting point for the results of this paper, as outlined in Section~\ref{sec:intro}.

\section{Algebraic framework for product representations}\label{Sec:ProdRep}

In this section we set up our general algebraic-categorical framework.
We assume given a variety  $\duplicateV$ of algebras for which 
we desire a product representation theorem,
and that $\baseV=\mathbb{V} (\CN)$
is a well-behaved and well-understood variety  on which
 we want to base our representation for $\duplicateV$. 
We aim to realise $\duplicateV$ as a variety $\mathbb{V}(\CM)$, 
where $\CM$ is obtained from $\CN$, in the manner 
outlined in Section~\ref{sec:intro}, 
by means of a set $\Gamma$ of pairs of 
terms in the language of   $\baseV$,
except that now do not restrict to singly-generated varieties.

The set $\Gamma$ is used to build 
a product structure
$\M \cong \fnt{P}_{\Gamma}(\N)$ of each algebra $\N\in \CN$. 
 We then seek to show that  $\baseV:=\mathbb{V}(\CN)$ and 
$\mathbb V(\fnt{P}_{\Gamma}(\CN))$ are categorically equivalent, 
with the second variety  being what we call a duplicate of the first 
(the formal definition  is given below).
Two extreme cases naturally arise here: 
$\CN$ is already  our base variety $\baseV$  or $\CN$  may  
contain a single algebra $\N$.  
The former case will arise in practice when $\baseV$ is not finitely generated, 
as occurs for example when $\baseV$ is~$\CL$ or~$\CLU$.
Our   programme will, however,  yield the most powerful results 
in the latter case and when, better still,  
we can show that $\duplicateV$ is generated by  $\fnt{P}_{\Gamma}(\N)$,
for some choice of~$\Gamma$.
In these circumstances Theorem~\ref{thm:GeneralProductEquiv} tells us that 
a product representation of a generator for $\duplicateV$ lifts to a product representation 
applicable to the entire equational class~$\duplicateV$, 
and that this lifting operates functorially.
We then have a very tight relationship  between $\baseV = \mathbb{V}(\N)$
and $\duplicateV= \mathbb{V}(\fnt{P}_{\Gamma}(\N))$;
indeed these varieties are equivalent as categories.  
This is exactly what happens, as we shall
demonstrate later, for many much-studied varieties, and it retrospectively vindicates 
the emphasis in much of the literature  (see for example 
\cite{Gin,GinMod,F91,FKlgen,AA1,RF,AA2}) on individual bilattice-based algebras
as opposed to the classes they generate:  algebraic information not visible 
at the level of the generator becomes instantly accessible, leading to a much richer theory.       

Let $\CN$ be a class of  $\Sigma$-\-alg\-ebras, 
where $\Sigma$ is some algebraic language and let  $\mathbb{V}(\CN)$ be the variety generated by $\CN$. 
Let $\Gamma$ be a set of pairs of $\Sigma$-terms 
such that, for each $(t_1,t_2)\in\Gamma$, 
there exists $n_{(t_1,t_2)}\in\{0,1,\ldots\}$ such that $t_1$ and $t_2$ 
are terms on $2n_{(t_1,t_2)}$ variables.  
We shall  view 
$\Gamma$ as playing  the role of an 
algebraic language for a family of algebras  $\fnt{P}_{\Gamma}(\A)$
($\A \in \mathbb{V}(\CN)$), where the arity of 
 $ (t_1,t_2)\in\Gamma $ is $n_{(t_1,t_2)}$.
We write $[t_1,t_2]$ 
 when we are viewing $(t_1,t_2)$ as belonging to~$\Gamma$, \textit{qua}
language, rather than as a pair of terms from the original language.
 Specifically  we define, for $\A \in \mathbb{V}(\CN)$,  
\[
\fnt{P}_{\Gamma}(\A)
=(A\times A; \{ [t_1,t_2]^{\fnt{P}_{\Gamma}(\A)} \mid (t_1,t_2)\in\Gamma \}),
\]
where, writing 
$n=n_{(t_1,t_2)}$,  
    the operation
$[t_1,t_2]^{\fnt{P}_{\Gamma}(\A)}\colon (A\times A)^{n}\to A\times A$ 
is  given by
\begin{multline*}
[t_1,t_2]^{\fnt{P}_{\Gamma}(\A)}
((a_1,b_1),\ldots,(a_{n},b_{n}))=
(t_1^{\A}(a_1,b_1,\ldots,a_{n},b_{n}),t_2^{\A}(a_1,b_1,\ldots,a_{n},b_{n})),
\\
\text{for }(a_1,b_1),\ldots,(a_{n},b_{n})\in A\times A.
\end{multline*}

We let  $\fnt{P}_{\Gamma}(\CN)$ denote the class of algebras of the form 
$\fnt{P}_{\Gamma}(\N)$, for $\N \in \CN$.  
 It is straightforward to check that 
$\fnt{P}_{\Gamma}(\mathbb{V}(\CN))$ is contained in $\mathbb{V}(\fnt{P}_{\Gamma}(\CN))$. 
We claim that the assignment  $\A\mapsto\fnt{P}_{\Gamma}(\A)$ (on objects)
and $h  \mapsto h \times h$ (on morphisms) defines
a functor $\fnt{P}_{\Gamma}\colon   \mathbb{V}(\CN) 
\to  \mathbb{V}(\fnt{P}_{\Gamma}(\CN))$.
We need to
confirm that $\fnt{P}_{\Gamma}$ is well defined on morphisms. 
Take $\A,\B\in \mathbb{V}(\CN)$ 
and $h\colon \A\to\B$ a homomorphism.  
Since the operations in $\fnt{P}_{\Gamma}(\A)$ and $\fnt{P}_{\Gamma}(\B)$ are constructed using $\Sigma$-terms 
 $h\times h\colon A\times A\to B\times B$ is indeed a homomorphism from $\fnt{P}_{\Gamma}(\A)$  to $\fnt{P}_{\Gamma}(\B)$.
It is routine to check 
that $\fnt{P}_{\Gamma}$ is a functor and is faithful.  

We introduce 
the following notation.
Given  a set $X$ we let $\delta^{X} \colon X\to X\times X$ 
be the diagonal map given by  
$\delta^{X} (x)=(x,x)$ and  let $\pi_1^{X},\pi_2^{X}\colon X\times X\to X$
 be the projection maps; we  suppress the  label when no ambiguity would arise. 

We are  now ready to give an  important definition. 
Fix a class $\CN$  of $\Sigma$-algebras that generates a variety 
$\baseV$  and let $\Gamma$ 
be a set of pairs of terms as specified above.
We say that the variety 
 $\duplicateV=\mathbb{V}(\fnt{P}_{\Gamma}(\baseV))$ 
 is \emph{a duplicate of}  $\baseV$ (in symbols $\baseV\dupl\duplicateV$) 
if $\Gamma$   \emph{duplicates} $\CN$.  By  the latter  we  mean  that
the following conditions on~$\CN$ and~$\Gamma$ are satisfied:
\begin{newlist}
\item[(L)] for each $n$-ary operation symbol $f\in\Sigma$ and $i\in\{1,2\}$ 
there exists an $n$-ary $\Gamma$-term $t$ such that
$\pi^{N}_i\circ t^{\fnt{P}_{\Gamma}(\N)}\circ(\delta^{N})^{n}=f^{\N}$ 
for each $\N\in\CN$;  
\item[(M)] there exists a binary $\Gamma$-term $v$ such that 
\[
v^{\fnt{P}_{\Gamma}(\N)}((a,b),(c,d))=(a,d)  \quad
 \text{ for $\N\in\CN$ and $a,b\in N$;} 
\] 
\item[(P)] there exists a unary $\Gamma$-term $s$ such that \[s^{\fnt{P}_{\Gamma}(\N)}(a,b)=(b,a) \quad \text{  for  
$\N\in\CN$ and 
$a,b\in N$;}
\]
\end{newlist}
Here L,  M and P have the connotations of  language, merging and permutation. 
The role of the term $v$ in 
(M) is to merge pairs and that  of term $s$ in
(P) is to permute the coordinates.
Therefore, if $\N\in\CN$ and $S$ is a subset of $\fnt{P}_{\Gamma}(\N)$ that is closed under $v$, 
 then $\pi_1^{N}(S)=\pi^N_2(S)$. 
If $S$ is closed under $s$, then 
$S=\pi_1^N(S)\times\pi_2^N(S)$. 
It is worth observing that,  if $\Gamma$ satisfies 
(P), then (L) is equivalent to the weaker condition
\begin{newlist}
\item[$(\text{L}'\,)$]
 for each $n$-ary operation symbol $f\in\Sigma$  there exist 
an $n$-ary $\Gamma$-term $t$ and $i\in\{1,2\}$ such that
$\pi^{N}_i\circ t^{\fnt{P}_{\Gamma}(\N)}\circ(\delta^{N})^{n}=f^{\N}$ 
for each $\N\in\CN$.  
\end{newlist}

The algebraic language determined by $\Gamma$ is  obtained by means of the pairs of terms in $\Sigma$. 
Condition (L) works in the reverse direction, as a method to obtain $\Sigma$ from terms in $\Gamma$. 
In Section~\ref{Sec:Conclusion} we elucidate the connection between 
product representation  and term-equivalence.

Illustrations of the duplication mechanism, for various base varieties and with 
a variety   
of duplicators $\Gamma$,  are given in succeeding sections.
We shall thereby bring many varieties within the scope of our main result, Theorem~\ref{thm:GeneralProductEquiv}.
Whether or not an algebra $\M$ on a universe $N \times N$ can be obtained 
as  a duplicate of some $\N$ with universe~$N$
will of course depend on
whether $\Gamma$, satisfying (L), (M) and (P), can be found so that the  
operations of $\M$ and $\N\odotG{\Gamma}  \N$ match up.  
See Example~\ref{ex:DM}  for an illustration of obstacles to duplication.

\begin{thm}\label{thm:GeneralProductEquiv}
Assume that  $\Gamma$ duplicates a class $\CN$
and let $\baseV=\mathbb{V}(\CN)$. 
Then the functor $\fnt{P}_{\Gamma}\colon \baseV\to \duplicateV$
sets up a categorical equivalence between 
 $\baseV$ and 
its duplicate $\duplicateV=\mathbb{V}(\fnt{P}_{\Gamma}(\CN))$.
\end{thm}
\begin{proof}
As we observed above, $\fnt{P}_{\Gamma}$ is a well-defined  and faithful functor. 
We only need to check that it is full and dense
on $\duplicateV=\mathbb{V}(\fnt{P}_{\Gamma}(\CN))$.
 To simplify notation, during this proof 
we write $\fnt{P}$ instead of~$\fnt{P}_{\Gamma}$.  

We first show that $\fnt{P}$ is full. 
Let $\A,\B\in\baseV$ and let $\psi\colon\fnt{P}(\A)\to\fnt{P}(\B)$
be a homomorphism.  
Let $h\colon\A\to\B$ be defined by $h=\pi_1^{B}\circ \psi\circ\delta^{A}$. 
We shall show that $h$ is a homomorphism and $\fnt{P}(h)=\psi$. 
By (P), we also have $h=\pi_2^{B}\circ \psi\circ\delta^{A}$.
By (M), there is a $\Gamma$-term $v$ such 
that $v^{\fnt{P}(\N)}((c,c),(d,d))=(c,d)$ for each $\N\in\CN$ and $c,d\in\N$.
Since $\A,\B\in\baseV$, 
the same equation is valid in $\A$ and $\B$. 
Hence
\begin{multline*}
\psi(a,b)=\psi(v^{\fnt{P}(\A)}((a,a),(b,b)))=v^{\fnt{P}(\B)}(\psi(a,a),\psi(b,b))\\
=(\pi_1^{\fnt{P}(\B)}(\psi(a,a)),\pi_2^{\fnt{P}(\B)}(\psi(b,b)))=(h(a),h(b)),
\end{multline*}
that is, $\psi=h\times h$.

Now let $f\in\Sigma$ be an $n$-ary operation symbol. By (L), there exist 
 an $n$-ary $\Gamma$-terms $t_1$ and $t_2$  
such that $\pi_i^{N}\circ t_i^{\fnt{P}(\N)}\circ (\delta^{N})^n=f^{\N}$ for 
$\N\in\CN$ and $i\in\{1,2\}$. 
Moreover there is a $\Gamma$-term $w$  such that
\[
w^{\fnt{P}(\N)}=v^{\fnt{P}(\N)}
\bigl(t_1^{\fnt{P}(\N)},t_2^{\fnt{P}(\N)}\bigr)
=f^{\N}\times f^{\N}
\]
for $\N\in\CN$, the corresponding
statement  holds also for each $\C$ that belongs to $\baseV$.
Hence, for $a_1,\ldots,a_n\in A$,
\begin{align*}
h(f^{\A}(a_1,\ldots,a_n))
&=\pi_1^{B}\circ \psi\circ\delta^{B}(f^{\A}(a_1,\ldots,a_n))\\
&=\pi_1^{B}(\psi((f^{\A}\times f^{\A})((a_1,a_1),\ldots,(a_n,a_n))))\\
&=\pi_1^{B}(\psi(w^{\fnt{P}(\B)}((a_1,a_1),\ldots,(a_n,a_n))))\\
&=\pi_1^{B}(w^{\fnt{P}(\B)}(\psi(a_1,a_1),\ldots,\psi(a_n,a_n)))\\
&=\pi_1^{B}((f^{\B}\times f^{\B})(\psi(a_1,a_1),\ldots,\psi(a_n,a_n)))\\
&=\pi_1^{B}((f^{\B}\times f^{\B})((h(a_1),h(a_1)),\ldots,(h(a_n),h(a_n))))\\
& =f^{\B}(h(a_1),\ldots,h(a_n)).
\end{align*}
This concludes the proof that $\fnt{P}$ is full.

It remains to show that $\fnt{P}$ is dense.
For every set of algebras  $\mathcal{K}\subseteq\baseV$,
it is easy to see that $\prod\fnt{P}(\mathcal{K})$ is isomorphic to $\fnt{P}(\prod\mathcal{K})$. 
Now let $\C\in\CN$ and let $\B\in\duplicateV$ 
be such that $\B$ is a subalgebra of $\A = \fnt{P}(\C)$. 
By (L), $\pi_1^A(B)$ and $\pi_2^A(B)$ are the universes of subalgebras $\C_1$ and $\C_2$ of $\C$.
By (P), $\pi_1^A(B)=\pi_2^A(B)$, hence $\C_1=\C_2$. 
By (M), $B=\pi_1^A(B)\times\pi_2^A(B)$. 
It follows that $\B\cong\fnt{P}(\C_1)$.

Let $\C\in\baseV$  and $\A\in \duplicateV$ and assume that 
$g\colon\fnt{P}(\C)\to \A$ is a surjective homomorphism. 
Consider 
 $q=g\circ\delta^{C}\colon \C\to \A$. 
We shall  show that $\theta:=\ker(q)$ 
is a congruence of $\C$. 
Let $f\in\Sigma$ be an $n$-ary operation 
and $(a_1,b_1),\ldots,(a_n,b_n)\in \theta$. 
We have already observed that by (L) and (M) there exists a term~$w$ such that
$w^{\fnt{P}(\C)}=f^{\C}\times f^{\C}$.
Hence
\allowdisplaybreaks
\begin{align*}
q( f^{\C}(a_1,\ldots,a_n) )
&=q( f^{\C}(a_1,\ldots,a_n), f^{\C}(a_1,\ldots,a_n))\\
&=q((f^{\C}\times f^{\C})((a_1,a_1),\ldots,(a_n,a_n)))\\
&=w^{\fnt{P}(\C)}(q(a_1,a_1),\ldots ,q (a_n,a_n))\\
&=w^{\fnt{P}(\C)}(q(b_1,b_1),\ldots ,q (b_n,b_n))\\
&=q((f^{\C}\times f^{\C})((b_1,b_1),\ldots(b_n,b_n)))=q( f^{\C}(b_1,\ldots,b_n)).
\end{align*}
Therefore $(f^{\C}(a_1,\ldots,a_n),f^{\C}(b_1,\ldots,b_n))\in\theta$. 
We claim that the map
 $\varphi\colon \fnt{P}(\C/\theta)\to\A$
 defined by $\varphi([a]_{\theta},[b]_{\theta})=g(a,b)$ 
is well defined and an isomorphism. 
First observe that if ${q(a_1)=q(a_2)}$ and $q(b_1)=q(b_2)$, then, for $i=1,2$, 
\[
g(a_i,b_i)=g(v^{\fnt{P}(\C)}((a_i,a_i),(b_i,b_i)))
=v^{\A}(g(a_i,a_i),g(b_i,b_i))
=v^{\fnt{P}(\A)}(q(a_i),q(b_i)).
\]
It follows that $\varphi$ is well defined.
 The fact that $\varphi$ is a homomorphism 
follows 
 from the fact that $h$ and $g$ 
are homomorphisms and the definition of the operations 
in~$\fnt{P}(\C/\theta)$.
\end{proof}

The structural information provided by a product representation 
for a variety  $\class V$ is of most value when additional   
properties of $\class V$ follow from it. 
Here we should distinguish between 
properties  which hold simply because there is a categorical equivalence between 
$\duplicateV$ and the base variety $\baseV$ and those which rely on the specific
algebraic form of the equivalence.  
Properties of the former type include those expressible in terms 
of injective homomorphisms, 
  which correspond to monomorphisms \cite[Section~14]{BD}, or surjective homomorphisms, which correspond to regular epimorphisms
 (note \cite[Proposition~7.37 and  Definition~7.71]{AHS}, \cite[Theorem~6.12]{BS}).  From this it follows 
easily that categorically equivalent varieties have isomorphic subvariety 
lattices---a fact well known to universal algebraists but hard to document explicitly.
In particular, assume 
that  $\Gamma$ duplicates a class of algebras $\CN$,
so that the functor $\fnt{P}_{\Gamma}$ is a categorical equivalence.
Then $\fnt{P}_{\Gamma}$  induces 
an isomorphism between the lattices
 of subvarieties of $\mathbb{V}(\CN)$
 and  of $\mathbb{V}(\fnt{P}_{\Gamma}(\CN))$. 
Moreover, $\Gamma$ also duplicates any 
 subvariety
  $\cat{K}$ of $ \mathbb{V}(\CN)$.

We now  record as a corollary to Theorem~\ref{thm:GeneralProductEquiv} 
further consequences  of the existence of a categorical equivalence.
In combination with our later results bringing product-representable 
varieties within the scope of Theorem~\ref{thm:GeneralProductEquiv}, this 
corollary provides a uniform derivation 
for results which have been proved 
 piecemeal in the literature in many 
 specific instances \cite{MPSV,RPhD, BJR11};  see also \cite{SW}.

 \begin{coro}\label{Cor:IrreducibleProjective}
Assume that $\Gamma$ duplicates a class of algebras $\CN$. 
The following statements hold for each ${\A\in\mathbb{V}(\CN)}$. 
\begin{newlist}
\item[{\upshape (a)}] $\Con  (\A)\cong \Con  (\fnt{P}_{\Gamma}(\A))$, 
where $\Con$ denotes the lattice of congruences of the corresponding algebra.
\item[{\upshape (b)}] 
$\A$ is  subdirectly irreducible if and only if $\fnt{P}_{\Gamma}(\A)$ is.
\end{newlist}
\end{coro}
\begin{proof}
(a) follows directly from the relation between congruences and regular epimorphisms, and (b) is a direct consequence of (a).
\end{proof}

 Any functor that determines a categorical equivalence preserves projective objects.  
Accordingly, if $\Gamma$ duplicates $\CN$ then
$\A$ is projective in $\mathbb{V}(\CN)$ if and only if
 $\fnt{P}_{\Gamma}(\A)$ is projective in $\mathbb{V}(\fnt{P}_{\Gamma}(\CN))$.
However, categorical equivalences do not always preserve free objects. 
Nonetheless, 
the following result tells us how to use $\fnt{P}_{\Gamma}$ to describe 
free objects 
in $\mathbb{V}(\fnt{P}_{\Gamma}(\CN))$ when those in $\mathbb{V}(\CN)$ are known.. 
Results of this type were obtained for distributive bilattices 
in~\cite[Section~8]{CPOne} using natural duality techniques. 
Here we see that they stem from the product representation, 
independently of the existence or not of a natural duality.

\begin{prop}\label{Prop:FreeAlgebra}
Let $X$ be a set, $\CN$ a class of algebras with the same language and 
$\baseV=\mathbb{V}(\CN)$ be the variety generated by $\CN$. 
If $\Gamma$ duplicates $\CN$ and $\duplicateV=\mathbb{V}(\fnt{P}_{\Gamma}(\CN))$, 
then  $\fnt{F}_{\duplicateV}(X)$, the  $\duplicateV$-free algebra over $X$, 
is isomorphic to the algebra $\fnt{P}_{\Gamma}(\fnt{F}_{\baseV}
(X\times\{0,1\}))$ 
and the isomorphism is obtained by the identification $x\mapsto ((x,0),(x,1))$ for 
$x\in X$, where $\fnt{F}_{\baseV}(X\times\{0,1\})$ is the $\baseV$-free algebra  over $X\times\{0,1\}$.
\end{prop}

\begin{proof}
It is easy to see that $\{\, ((x,0),(x,1))\mid x\in X\,\}$ is a set of generators of  the algebra $\fnt{P}_{\Gamma}(\fnt{Fr}_{\CN}(X\times\{0,1\}))$. 

Now  let $\B\in\mathbb{V}(\fnt{P}_{\Gamma}(\CN))$ and consider a map 
$f\colon \{\,((x,0),(x,1))\mid x\in X\,\}\to \B$. 
By Theorem~\ref{thm:GeneralProductEquiv}, there exists 
$\A\in\mathbb{V}(\CN)$ such that $\B\cong\fnt{P}_{\Gamma}(\A)$. 
Let us identify $\B$ with $\fnt{P}_{\Gamma}(\A)$. Let $g\colon X\times\{0,1\}\to \A$ 
be the map defined by $g(x,i)=\pi_i^{A}(f((x,0),(x,1)))$ for $i=0,1$ and $x\in X$.  
There exists a unique homomorphism
$\bar{g}\colon \fnt{Fr}_{\CN}(X\times\{0,1\})\to \A$ 
such that $g(x,i)=\bar{g}(x,i)$ for $(x,i)\in X\times\{0,1\}$. 
Let $h=\fnt{P}(\bar{g})\colon \fnt{P}_{\Gamma}(\fnt{Fr}_{\CN}(X\times\{0,1\}))\to \fnt{P}_{\Gamma}(\A) $. 
For $x\in X$,
\[
h((x,0),(x,1))=\fnt{P}(\bar{g})((x,0),(x,1)=(\bar{g}(x,0),\bar{g}(x,1))=(g(x,0),g(x,1)) 
=f((x,0),(x,1)).
\]
That is, $h$ extends $f$. 
\end{proof}

\section{Duplication in action:  interlaced 
and distributive bilattices revisited}\label{Sec:BilatticeProd}

We fix some notation. Let $\Sigma$ be a language and $f$ be an 
$n$-ary function symbol in $\Sigma$, then 
for each $m\geq n$ and $i_1,\ldots, i_n\in\{1,\ldots m\}$, we denote by  
$f^{m}_{i_1\cdots i_n}$  the $m$-ary term 
\[
f^{m}_{i_1\ldots i_n}(x_1,\ldots,x_m)=f(x_{i_1},\ldots,x_{i_m}).
\]
Similarly, $x^{m}_i$ denotes the $m$-ary term that selects the $i$th variable:  
${x^{m}_{i}(x_1,\ldots, x_m)=x_i}$.
For example, let  $\Sigmavar {\CL_u}=\{\vee,\wedge\}$  be the language of lattices.  
Then $\vee^{4}_{13}$ denotes the term 
$\vee^{4}_{13}(x_1,x_2,x_3,x_4)=x_1\vee x_3$.

Now consider the set of $\Sigmavar {\CL_u}$-pairs of terms 
\[
\Gammavari  {\BL_u}  
=\bigl\{
(\vee^{4}_{13},\wedge^{4}_{24}),
(\wedge^{4}_{13},\vee^{4}_{24}),
(\vee^{4}_{13},\vee^{4}_{24}),
(\wedge^{4}_{13},\wedge^{4}_{24}),
(x^{2}_2,x^{2}_1)
\bigr\}.
\]
We name  
\[
\vee_t=[\vee^{4}_{13},\wedge^{4}_{24}],
\ 
\wedge_t=[\wedge^{4}_{13},\vee^{4}_{24}], 
\ 
\vee_k=[\vee^{4}_{13},\vee^{4}_{24}], 
\
\wedge_k=[\wedge^{4}_{13},
\wedge^{4}_{24}], 
\mbox{ and }\neg=[x^{2}_2,x^{2}_1],
\]
to match up our newly-created operations with those in the language of $\BLU$.
We can clearly see that $\fnt{P}_{\Gammavari{\BL_u}}(\Lalg)=\Lalg\odot\Lalg$. 
The Product Representation Theorem for unbounded interlaced bilattices implies
that every $\A\in\BLU$ is isomorphic to $\fnt{P}_{\Gammavari{\BL_u}}(\Lalg)$ for some $\Lalg\in \CLU$.
Thus
$\mathbb{V}({\fnt{P}_{\Gammavari{\BL_u}}(\CLU)})=\BLU$. 
Moreover, it is known that $\fnt{P}_{\Gammavari{\BL_u}}$ determines a categorical equivalence \cite{BR11}. This follows directly from $\mathbb{V}(\fnt{P}_{\Gammavari{\BL_u}}(\CLU))=\BLU$ and Theorem~\ref{thm:GeneralProductEquiv}, by simply 
observing that $\Gammavari {\BL_u}$  duplicates 
$\CLU$.
Indeed, it is easy to see that $\Gammavari {\BL_u}$ satisfies (L) and (P). 
Observe too that, for  $\Lalg\in\CLU$
 and $a,b\in L$,
\[
((a,b)\wedge_k((a,b)\vee_t( c,d)))\vee_k((c,d)\wedge_k((a,b)\wedge_t(c,d)))=(a,b).
\]
 Hence the term 
$v(x,y)=(x\wedge_k (x\vee_{t}y))\vee_k(y\wedge_k(x\wedge_t y))$  satisfies~(M).

We can easily add bounds:  let 
$\Gammavari{\mathbf b}=\{(0,1),(1,0),(0,0),(1,1)\}$;
this is  a set of  pairs of terms  
 in the language of $\CL$ and we may then take    
$\Gammavari {\BL}=\Gammavari {\BL_u}\cup\Gammavari{\mathbf b}$.
It is straightforward to check that $\Gammavari {\BL}$ satisfies conditions (L), (M) and (P). 
Therefore $\fnt{P}_{\Gammavari {\BL}}$ determines a categorical equivalence 
between $\CL$ and $\mathbb{V}(\fnt{P}_{\Gammavari {\BL}}(\CL))=\BL$.

Lattices are not a finitely generated variety, and our product representation 
for $\BLU$ over $\CLU$ had to take $\CN=\CLU$.  For the variety $\DBU$ distributive 
bilattices the situation is different: the obvious base variety to use,
(unbounded) distributive lattices, is finitely generated.  
We now fit the product representation for $\DBU$
into our general scheme, using Theorem~\ref{thm:GeneralProductEquiv} 
as it applies to a singly generated variety.

We denote by $\CCD$ and $\CCD_u$  the varieties of 
bounded distributive lattices and of  unbounded distributive lattices, respectively.
We let $\twovar{\CCD}  $, respectively $\twovar{\CCD_u}$, 
denote the two-element algebra in $\CCD$,
respectively $\CCDU$.  In both cases we take the underlying set to have elements 
$0,1$, with $0 < 1$ and denote the corresponding non-strict order by $\leq$.
The following  well-known facts will be important later: 
\[
\CCDU = \HSP(\twovar{\CCD_u}) =\ISP(\twovar{\CCD_u}) 
\quad \text{and} \quad
\CCD = \HSP(\twovar{\CCD}) = \ISP(\twovar{\CCD}).
\]
By Theorem~\ref{thm:GeneralProductEquiv}, it follows that
$ \HSP(\fnt{P}_{\Gammavari{\BL_u}}(\twovar{\CCD_u}))
 =\ISP(\fnt{P}_{\Gammavari{\BL_u}}(\twovar{\CCD_u}))$.
Letting 
\[
\fourvar{\DB_u}
= (\{0,1\}^{2}; \vee_t,\wedge_t, \vee_t,\wedge_t, \neg)\coloneq\fnt{P}_{\Gamma_{\BL_u}}(\twovar{\CCD_u}),
\]  
we see that $\CCDU$ is categorically equivalent to 
$\ISP(\fourvar{\DB_u})=\HSP(\fourvar{\DB_u})$. 
So it remains to characterise the variety $\HSP(\fourvar{\DB_u})$.
 This is known to be the variety $\DBU$ of distributive bilattices, 
 that is, bilattices such that each of the four operations distributes over each of the other three.  
 Moreover, in \cite[Proposition~5.1]{CPOne}, we presented a proof that 
 $\ISP(\fourvar{\DB_u})=\DB_u$ that is independent of the product representation. 
Therefore $\CCDU\dupl \DBU$. 
Similarly, it follows that $\CCD\dupl\DB$, where $\DB$ is the variety of bounded distributive bilattices.

\section{Bilattices with conflation}\label{sec:Conflation}

Involutory operations are 
often
added to lattice-based 
varieties, and hence to bilattice-based varieties too, to provide 
algebraic models  which  capture  more than just notions of truth and knowledge.
We have already built in an involutory operation $\neg$ to model logical negation 
but wish also,  here and in
Section~\ref{Sec:Trilattices} too,  to allow for involutions which serve to model, 
for example, what is not known.  
To fit their intended interpretations,
such operations need to act  appropriately
with respect to the underlying order structures.
As we shall see, adding such operations influences our choice of base variety.
So we begin this section with a discussion of two  
finitely generated  
varieties, De Morgan lattices
and De Morgan algebras,
we have not encountered previously in this paper.  These will 
prove to be valuable as base varieties in due course.  In addition they
enable us to provide further illustration of the concept of duplication.

\begin{ex}[De Morgan algebras and De Morgan lattices]  
\label{ex:DM}
In Section~\ref{Sec:BilatticeProd} 
we encountered a four-element bounded bilattice,  obtained by 
duplicating the two-element bounded lattice.
We shall now compare this with another  four-element algebra, that which generates (as a prevariety) 
the variety $\DM$ of De Morgan algebras (a good reference is \cite[Chapter~XI]{BD}). 
 An algebra $\A=(A;\vee,\wedge,\sim,0,1)$ belongs to  
$\DM$ if $(A;\vee,\wedge,0,1)\in \CCD$ 
and $\sim$ is an order-reversing involution.  The variety is generated, as a prevariety, by 
 the algebra 
$\fourvar{\DM} $,   the De Morgan algebra whose  $\CCD$-reduct  
is $\twovar{\CCD}^2 $ and whose negation  
$\sim$
 interchanges the bounds  and fixes the other two elements.

We may ask whether $\fourvar {\DM}$ is a duplicate of a two-element algebra
in some naturally related base variety $\mathbb V(\N)$.
It is a consequence of Theorem~\ref{thm:GeneralProductEquiv} 
that this could only occur if $\DM$ were categorically equivalent to 
$\mathbb V(\N)$.  We note that $\DM$ is not categorically equivalent either to 
$\CCD$ or to $\CB$, the variety of Boolean algebras 
(the subvariety lattice of $\DM$ is not isomorphic to that of $\CCD$ or of~$\CB$)).
It is however quite simple to construct 
sets $\Gamma$ of pairs of terms in the languages 
 $\Sigma_{\CCD} = \{\lor,\land,0,1\}$ of~$\CCD$ or
 $\Sigma_{\CB} = \{\lor,\land,', 0,1\}$ of~$\CB$ such that 
  $\fourvar {\DM}\cong \fnt{P}_{\Gamma}(\twovar {\CCD})$ or
  $\fourvar {\DM}\cong \fnt{P}_{\Gamma}(\twovar {\CB})$.
We might take for example $\Gamma$ to be $\Gamma_1$ or~$\Gamma_2$, where
 \begin{align*}
\Gamma_{1} &=\{(\wedge^2_{13},\wedge^2_{24}),(\vee^2_{13},\vee^2_{24})), 
((')^2_2,(')^2_1),
(0,0), (1,1)\};  \\
 \Gamma_{2}&=\{(\wedge^2_{13},\vee^2_{24}),(\vee^2_{13},\wedge^2_{24}),
 (x^2_2,x^2_1),
 (0,1),(1,0)\} .
\end{align*}
It is easy to check that $\fourvar {\DM}\cong \fnt{P}_{\Gamma_1}(\twovar {\CCD})\cong \fnt{P}_{\Gamma_2}(\twovar {\CB})$. 
However
 $\Gamma_1$ satisfies (L$'$) but not (P), and $\Gamma_2$ satisfies 
 (P) but not (L$'$). So neither $\Gamma_1$ nor $\Gamma_2$ is a duplicator.
 
The unbounded
analogue of $\DM$ is the variety $\DMU$ of De Morgan lattices, 
that is, an algebra $\A=(A;\vee,\wedge,\sim)\in\DMU$ if $(A;\vee,\wedge)\in \CCDU$ and 
 and $\sim$ is an order-reversing involution.  
The variety $\DMU$ coincides with  $\ISP(\fourvar {\DM_u} )$, where 
$ \fourvar {\DM_u} = \bigl( \{ 0,1\}^{2}; \vee,  \wedge ,\sim\bigr);$ 
is the $\{0,1\}$-free reduct of $\fourvar {\DM}$
\cite[Theorem~1]{Kalman}.  
The variety $\DMU$ does not arise by duplicating either $\CCDU$ or the variety of Boolean lattices.
\end{ex}

We conclude from the above example that we should  regard the varieties $\DM$ and $\DMU$ as `atomic':  
their members are not built from simpler components by duplication.  
We shall see that 
they do have an important role to play as base varieties.

We now turn to the main topic of this section.
We consider  expansions of the varieties $\DBU$ and $\DB$ of 
(unbounded and bounded) distributive bilattices obtained by adding 
a unary operator $-$ called \emph{conflation} and required to act as an 
endomorphism for the truth lattice structure and a dual endomorphism 
for the knowledge lattice structure. 
 Customarily it has been assumed that
$-$ is an involution and that it commutes with $\neg$.   In this case
we denote the resulting expansion of~$\DBU$ by~$\DBCU$ and by $\DBC$ the expansion of $\DB$.

As indicated above,  the variety $\DBCU$ consists of algebras 
$ (A; \lor_t,\land_t,\lor_k,\land_k,\neg, -)$ 
for which the reduct without  $-$  belongs to~$\DBU$ and
  $-$ is an involution preserving $\leq_t$, reversing $\leq_k$ 
  and commuting with~$\neg$. 
The class $\DBC$ of bounded distributive 
  bilattices with conflation, where $-$ and~$\neg$ commute,  
 is defined in a similar way.  
The  product representation for $\DBCU$
was first presented in \cite[Theorem 8.3]{FKlkids}. 
What we shall do is to demonstrate how this product representation 
 for $\DBCU$,  and also that for $\DBC$ likewise, is a  particular case of our 
Theorem~\ref{thm:GeneralProductEquiv}.  
Indeed we shall see that the properties of conflation
essentially dictate what the base variety should be.  

Until further notice we work with $\DBCU$.
 We first note that we would expect to use a class 
having a reduct in unbounded distributive lattices, since that will already provide 
a  set 
$\Gammavari{\DBU}$ that satisfies (L), (M) and (P), and 
will allow us to represent the $\DBU$-reducts of algebras in $\DBCU$. 
To obtain the conflation operation in a product representation 
we need a pair of terms $(t_1,t_2)$ such that $[t_1,t_2]$ interprets as 
an involution that reverses the $k$-order. 
This forces $t_{1}(a\wedge b)=t_1(a)\vee t_1(b)$.  
This cannot be obtained with $\{\vee, \wedge\}$-terms since these 
preserve the order. 
So it is natural to add an involution to the language of $\CCDU$ 
to obtain the base variety we require.
An obvious   candidate is  to hand, namely the variety $\DMU$ of De Morgan lattices.
%
It is easy to see that 
$\Gammavari {\DBC_u}
=\Gammavari {\BL_u}\cup\{(\sim^{2}_2,\sim^{2}_1)\}$ 
satisfies (L),
since 
\[
\pi^{\four}_1 ([\sim^{2}_2,\sim^{2}_1]^{\fnt{P}_{\Gammavari {\DBC_u}}(\fourvar {\DM_u})}(a,a))
=\pi^{\four}_1 ({\sim}a, {\sim}a)={\sim}a
\]
for every $a\in \fourvar {\DM_u}$ and $\Gammavari{\BL_u}$ satisfies (L$'$).
Conditions (M) and  (P)  hold because they hold for  $\Gammavari {\BL_u}$. 
Therefore $\Gammavari {\DBC_u}$ duplicates $\DMU$.

To be able to apply Theorem~\ref{thm:GeneralProductEquiv}, 
it now only remains to prove that the variety $\DBCU$ 
coincides with $\mathbb{V}(\fnt{P}_{\Gammavari {\DBC_u}}(\DMU))$.
It is easy to see  that 
$\Svar{\DBC_u}:=\fnt{P}_{\Gammavari {\DBC_u}}(\fourvar {\DM_u} )$ 
is a bilattice with conflation and hence that 
$\mathbb{V}(\fnt{P}_{\Gammavari {\DBC_u}}(\DMU))
=\mathbb{V}(\fnt{P}_{\Gammavari {\DBC_u}}(\fourvar {\DM_u} ))\subseteq \DBCU$.
The reverse inclusion follows from the following stronger result.

\begin{prop} \label{sep-confl}
 $\DBCU = \ISP(\Svar{\DBC_u})$.
\end{prop}
\begin{proof} 
Let $\A\in\DBCU$ and take $a\neq b$ in $A$. By \cite[Proposition 5.1]{CPOne}, 
there exists a $\DBU$-homomorphism 
$h\colon \A\to \fourvar{\DB_u}$ such that $h(a)\neq h(b)$. 
Denote by $h_1$ and $h_2$ the unique maps from 
 $\A$ into $\{0,1\}$ such that $h(c)=(h_1(c),h_2(c))$, for $c\in A$.  
Define $h'\colon \A\to\Svar{\DBCU}$  by 
 \[
 h'(c)=\bigl( ( h_1(c),(1-h_2(-^{\A}c))),(h_2(c),(1-h_1(-^{\A}c)))\bigr)
 \]
for $c\in A$. 
Clearly $h'(a)\neq h'(b)$. To prove that $h'$ is a $\DBCU$-homomorphism, 
first observe that, since $h$ is a $\DBU$-homomorphism, 
\begin{alignat*}{2}
h_1(c\vee_t d)&=h_1(c\vee_k d)=h_1(c)\vee h_1 (d), \quad& h_1(c\wedge_t d)=h_1(c\wedge_k d)=h_1(c)\wedge h_1 (d);\\
h_2(c\vee_t d)&=h_2(c\wedge_k d)=h_2(c)\wedge h_2 (d),  \quad &h_2(c\wedge_t d)=h_2(c\vee_k d)=h_2(c)\vee h_2 (d)
\end{alignat*}
and $h_1(c)=h_2(\neg c)$.
It is then easy to see that $h'$ is a 
$\DBU$-homomorphism. 
Moreover,
\begin{align*}
h'(-^{\A}c)&=\bigl((h_1(-^{\A}c),(1-h_2(c))) , (h_2(-^{\A}c),(1-h_1(c)))\bigr)\\
&=\bigl(\sim (h_2(c),1-h_1(-^{\A}c)),\sim( h_1(c),(1-h_2(-^{\A}c))) \bigr)\\
&=[\sim^{2}_2,\sim^{2}_1]^{\Svar{\DBC_u}}\bigl( ( h_1(c),(1-h_2(-^{\A}c))),(h_2(c),(1-h_1(-^{\A}c)))\bigr)\\
&=[\sim^{2}_2,\sim^{2}_1]^{\Svar{\DBC_u}}(h(c)).
\end{align*}
Hence $h'$ is a  $\DBCU$-homomorphism.
\end{proof}

The product representation for $\DBC$ is obtained in a similar way using the 
variety $\DM$ of De Morgan algebras as a base class and
 $\Gammavari{\DBC}=\Gammavari{\DBC_u}\cup\Gamma_{\mathbf{b}}$.

We note that neither the requirement that $-$ be an involution nor 
the assumption that it  commute with~$\neg$ has been driven by applications. 
 In \cite{TwoPlus} we relax these restrictions on conflation and 
provide a product representation and a natural duality for the resulting class.

\section{Trilattices}  \label{Sec:Trilattices}

Trilattices are, loosely, algebras with three sets of lattice operations, the idea
being to model information, truth and falsity.  
An introduction to the topic from  a logical standpoint can be found in \cite{SDT,SW}.

As with bilattices, inclusion of bounds is optional.  
For illustrative purposes 
we consider the unbounded case. 
To simplify notation a little we shall omit $_u$ subscripts from our symbolic names  for
trilattice and trilattice-based varieties. 
Thus a \emph{trilattice} is an algebra 
\[
\A = (A;\vee_t ,\wedge_t ,\vee_f ,\wedge_f ,\vee_i ,\wedge_i)
\]
such that its reducts
$\A_t= (A;\vee_t ,\wedge_t)$, $\A_f=(A;\vee_f ,\wedge_f)$ and 
$\A_i=( A; \vee_i,\wedge_i)$ are lattices.  
For any  trilattice $\A$ we let $\A_{t,i}$
 denote the bilattice reduct of $\A$ obtained by removing 
 the $f$-operation, and so on.

As with bilattices, at a minimum, an interlacing condition is required in order to obtain a worthwhile structure theory.  
In  Example~\ref{Ex:IntTri} we consider interlaced trilattices.  
Here we impose the stronger restriction
of distributivity, thereby  moving into the setting of finitely generated
varieties in which a particularly amenable structure theory becomes available.
We let $\TL$ denote the variety of (unbounded) distributive trilattices, 
that is, those trilattices in which all possible distributive laws 
hold amongst the six lattice operations.

 The following examples of trilattices  introduce  notation we need  shortly.  
$\two^{++}$, $\two^{+-}$, $\two^{-+}$, $\two^{--}\in\TL$ 
denote the  trilattices whose universe is $\{0,1\}$ and such that 
\begin{alignat*}{2}
\two_i^{++}=\two_i^{+-}=\two_i^{-+}&=\  \two_i^{--}=\twovar{\CCD_u},\\ 
\two_t^{++}=\two_t^{+-}=\two_f^{++}=\two_f^{-+}=\twovar{\CCD_u}, \mbox{ and \ }
&\two_t^{-+}=\two_t^{--}=\two_f^{+-}=\two_f^{--}=\twovar{\CCD_u}^{\partial}.
\end{alignat*}

There are various ways in which  one  might want involutory 
operations on trilattices to behave, depending on the desired  
interpretation.   
The involutions considered in~\cite[Definition 5.2]{SDT} 
and~\cite[Sections 3.2--3.4]{R13} are dual endomorphisms for one lattice reduct 
 and endomorphisms for the other two
reducts. So, a \emph{$v$-involution} 
 (where $v\in\{t,f,i\}$) is an involutory operation on a trilattice 
 that reverses the $v$-lattice reduct and preserves the other two reducts.
Let $\TL_t$, $\TL_{t,f}$ and $\TL_{t,f,i}$ denote the varieties of trilattices with $t$-involution,
with $t$- and $f$-involutions, and with $t$-, $f$-  and $i$-involutions,
 respectively.  
 Clearly these three varieties cover all the cases we need to consider.  
We shall assume that all the involutions which we include 
commute with each other.

As examples of trilattices with a single involution we note that
$\four^+$ and $\four^-$  are trilattices with $t$-involution $-_{t}$ having  universe $\{0,1\}^2$ when 
we define  
\begin{gather*}
\four^+_{t}=\four^-_{t}=\twovar{\CCD_u}\times\twovar{\CCD_u}^{\partial}, \quad  \four^+_{i}=\four^-_{i}=\four^+_f=\twovar{\CCD_u}\times\twovar{\CCD_u}, \quad
\four^-_f=\twovar{\CCD_u}^{\partial}\times\twovar{\CCD_u}^{\partial};\\
-_{t}(a,b)=(b,a).
\end{gather*}

Just as a single involution led to the construction of four-element
 trilattices from two-element ones, sixteen-element trilattices arise 
 naturally from four-element ones when two involutions come into play.
 We let $\Svar{\TL_{t,f}}$ denote the trilattice with $t$- and $f$-involutions 
 with universe $(\{0,1\}^{2})^2$ whose operations are defined as follows:
\begin{gather*} 
(\Svar{\TL_{t,f}})_t =(\fourvar{\DB_u})_{t}^2, 
\quad
(\Svar{\TL_{t,f}})_f=(\fourvar{\DB_u})_{k}\times(\fourvar{\DB_u})_{k}^{\partial}, \quad 
(\Svar{\TL_{t,f}})_i=(\fourvar{\DB_u})_{k}^2;\\ 
-_t (a,b)=(\neg^{\fourvar{\DB_u}}(a),\neg^{\fourvar{\DB_u}}(b)),\\
-_f (a,b)=(b,a).
\end{gather*}

And, finally, we can encompass  
three involutions.  Let $\bf 256$ be the trilattice whose universe is 
 $(\{0,1\}^4)^{2}$ with $t$,$f$ and $i$-involutions such that 
\begin{gather*}
{\bf 256}_{t}=(\Svar{\DBC_u})_t^2,\quad {\bf 256}_{f}=(\Svar{\DBC_u})_k^2,\quad  {\bf 256}_{i}=(\Svar{\DBC_u})_k\times(\Svar{\DBC_u})_k^{\partial};\\
-_t(a,b)=(\neg^{\Svar{\DBC_u}}(a),\neg^{\Svar{\DBC_u}}(b)), \quad -_f(a,b)=(-^{\Svar{\DBC_u}}(a),-^{\Svar{\DBC_u}}(b)),  \\
{-_i(a,b)=(b,a)}.
\end{gather*}

The following lemma is the stepping-off point  for further analysis of
 trilattices by the methods of this paper.   
\begin{lem} \label{lem:ISPTrilattices}
\begin{alignat*}{4} 
\text{\upshape (i)}&& \ \TL& =\ISP(\two^{++},\two^{+-},\two^{-+},\two^{--});
\qquad \qquad & \text{\upshape (iii)}&& \ \TL_{t,f} &=\ISP(\Svar{\TL_{t,f}});\\ 
\text{\upshape (ii)}&&\  \TL_t &=\ISP(\four^{+},\four^{-});  &
\text{\upshape (iv)}&&\ \TL_{t,f,i} &=\ISP({\bf 256}).
\end{alignat*} 
\end{lem}
\begin{proof}
Let $\A\in\TL$ and take $a\neq b$ in $A$.
Then there exists a lattice homomorphism $h \colon \A_i \to~\two$ 
such that $h(a)\neq h(b)$.
The assumed distributivity of the trilattice operations ensures that, 
 for each $\A\in\TL$, a  congruence of $\A_i$ 
is a congruence of  $\A$ (see~\cite[Proposition~3.13]{BR11} 
or~\cite[Proposition 2.2]{CPOne} for a simple proof). 
Hence $\ker(h)$ is a congruence of $\A$ and $|\A/{\ker(h)}|=2$. 
Therefore $\A/{\ker(h)}$ is necessarily  isomorphic to 
$\two^{++},\two^{+-},\two^{-+},$ or $\two^{--}$, and the proof of (i) is complete.

We now prove (ii).
Let $\B\in\TL_t$ and take  $a\neq b$ in $B$.   
Then $\B_{t,i}\in\DBU=\ISP(\fourvar{\DB_u})$.
 Therefore there exists a homomorphism $h\colon \B_{t,i}\to \fourvar{\DB_u}$ 
such that $h(a)\neq h(b)$. 
As before, $\ker(h)$ is also compatible with the $f$-lattice structure.
Then $\B/\ker(h)$ is a trilattice with four elements such that 
its $t,i$ and $t,f$ reducts are isomorphic to $\four_{\DBU}$. 
Therefore  $\B/\ker(h)$ is either isomorphic to $\four^+$ or 
to $\four^-$ and the result follows.

Now let $\C\in\TL_{t,f}$ and take $a,b\in C$ such that $a\neq b$. 
Then $\C_{t,i}\in\DBU=\ISP(\fourvar{\DB_u})$, so 
there exists  a homomorphism $h\colon \C_{t,i}\to\fourvar{\DB_u}$ 
such that 
$h(a)\neq h(b)$. Since $-_{f}$ preserves the $t$-order and the $i$-order,
 and it commutes with $-_t$, it follows that $h\circ (-_{f})$ 
 is also a homomorphism from 
$\C_{t,i}$ onto  $\fourvar{\DB_u}$. 
Then the map $g\colon \C\to \Svar{\TL_{t,f}}$ defined by 
$g(a)=(h(a),h(-_{f\,}a))$ is a homomorphism from $\C$ 
to $\Svar{\TL_{t,f}}$ that separates~$a$ and~$b$.

The proof of (iv) can be carried out  in a similar way to that of (iii).
\end{proof}

From  the definition  of $\Svar{\TL_{t,f}}$ it is easy to extract a duplicator 
$\Gammavari{\TL_{t,f}}$. 
Indeed, letting
\begin{multline*}
\Gammavari{\TL_{t,f}}= 
\bigl\{ 
 \bigl((\vee_t)^4_{13},(\vee_t)^4_{24}\bigr),   
\bigl((\wedge_t)^4_{13},(\wedge_t)^4_{24}\bigr),
\bigl((\vee_k)^4_{13},(\vee_k)^4_{24}\bigr),
\bigl((\wedge_k)^4_{13},(\wedge_k)^4_{24}\bigr),
\\
   \bigl((\wedge_k)^4_{13},(\vee_k)^4_{24}\bigr),  
   \bigl((\vee_k)^4_{13},(\wedge_k)^4_{24}\bigr),
  (\neg^2_{1},\neg^2_{2}),
(x^{2}_2,x^{2}_1)
\bigr\}
\end{multline*}
we obtain
$\Svar{\TL_{t,f}}=\fnt{P}_{\Gammavari{\TL_{t,f}}}(\fourvar{\DB_u})$.

Similarly,  from the definition of  ${\bf 256}$ we can obtain a duplicator $\Gamma_{\TL_{t,f,i}}$ for
 $\{{\Svar{\DBC_u}}\}$ and such that
 ${\bf 256}=\fnt{P}_{\Gammavari{\TL_{t,f,i}}}(\Svar{\DBC_u})$.
Therefore, Theorem~\ref{thm:GeneralProductEquiv} and 
Lemma~\ref{lem:ISPTrilattices} prove that $\DBU\dupl\TL_{t,f}$ and $\DBCU\dupl\TL_{t,f,i}$.

Similar results can be obtained for interlaced trilattices without the distributivity condition.
Some results on product representations for these more general classes of 
interlaced trilattices were  presented in \cite{R13} 
(see also Example~\ref{Ex:IntTri}).

\section{Bilattices with implication-like operations}\label{Sec:Residuated}

Bilattices with implication-like operations have been quite extensively
considered in the literature (see~\cite{BR13} and the references therein). 
A natural implication in an algebra with a lattice reduct arises as 
 the adjoint of the meet operation, if  this adjoint exists. 
 Given a lattice $\Lalg$, the operation $\to$ is the adjoint (or residuum) of
$\wedge$ if, for $a,b,c\in \Lalg$, 
\[
a\wedge b \leq c \Longleftrightarrow b \leq a\to c.
\]
An algebra $(A;\vee,\wedge,\to,0,1)$ such that $(A;\vee,\wedge,0,1)\in\CCD$ 
and $\to$ is the adjoint of $\wedge$ is a \emph{Heyting algebra} 
\cite[Chapter IX]{BD}. 
We  denote the variety of Heyting algebras by $\cat H$.

Any bilattice has two lattice reducts, and hence there are 
two natural candidates for implications: 
knowledge implication $\to_k$, the adjoint of $\land_k$, 
and truth implication $\to_t$, the adjoint of $\land_t$. 
Despite their definitions being so alike these  implications  exhibit different behaviour.  
As we shall see,  constants play an important role here.

\subsection*{Bilattices with  knowledge implication}\

Let $\krB$ denote the class of bounded bilattices whose knowledge 
lattice reduct is a Heyting algebra,  with the  implication
included in the language. 
More precisely, we consider algebras of the form
$\A=(A;\lor_t,\land_t,\lor_k,\land_k,\to_k,\neg, 0_t,1_t,0_k,1_k)$,
where the reduct omitting $\to_k$
 is a bilattice and $(\land_k,\to_k)$ is 
 an adjoint pair.
Then $(A;\lor_k,\land_k,\to_k,0_k,1_k)$ belongs to  $\cat{H}$. 
 We deduce that the class of bilattices with knowledge implication
 $\krB$ is a variety.
We shall show that $\krB$ is categorically  equivalent to~$\cat{H}$. 

We first show that the class of bilattices with  knowledge implication
 naturally arises  as a duplicate of $\cat{H}$. 
Let $\A=(A;\lor_t,\land_t,\lor_k,\land_k,\to_k,\neg, 0_t,1_t,0_k,1_k)\in\krB$. 
Then there exists $\Lalg=(L;\lor,\land,0,1)\in\CL$ such that 
$\A_{\BL}$,  the bilattice reduct of $\A$, is isomorphic to $\fnt{P}_{\Gammavari{\BL}}(\Lalg)=\Lalg\odot\Lalg$. 
We  identify $\A_{\BL}$ with $\fnt{P}_{\Gammavari{\BL}}(\Lalg)$. 
Since $(\land_k,\to)$ is an adjoint pair we have, for $a,b,c\in L$, 
\begin{multline*}
a\land b\leq  c \Longleftrightarrow (a,0)\land_k(b,0)\leq_k (c,0)
\Longleftrightarrow (b,0)\leq_k (a,0)\to_k (c,0)\\
\Longleftrightarrow b\leq \pi_{1}((a,0)\to_k (c,0)).
\end{multline*}
 Therefore, the operation $\to^{\Lalg}$, defined by $x\to^{\Lalg} y= \pi_{1}((x,0)\to_k (y,0))$, 
 is the  adjoint of~$\land$ and $(L;\lor,\land,\to^{\Lalg},0,1)\in\cat{H}$. 
Moreover, it follows that $(a,b)\to_k(c,d) = (a\to^{\Lalg}c , b\to^{\Lalg} d)$.
What we have actually proved is that the set
 \[
\Gamma_{\cat{H}}=\Gamma_{\BL}\cup\{(\to^{4}_{13}, \to^{4}_{24})\}
\] 
satisfies (L), (M)  and (P)  with respect to the language of $\cat{H}$. 
Now an application of Theorem~\ref{thm:GeneralProductEquiv} 
proves our claim that $\krB$ is categorically equivalent to $\cat{H}$.

In \cite{BR13}, the authors introduced Brouwerian 
bilattices and in \cite[Theorem~2.6]{BR13} they  
 presented a product representation for these. 
The base class for their product representation is the variety $\cat{BR}$ 
of Brouwerian  lattices (also known as generalised Heyting algebras);  
this is the variety of $0$-free reducts of Heyting algebras. 
 The product representation in \cite{BR13}  implicitly relies on a 
 duplicator different from ours, {\it viz.} 
 \[
\Gamma_{\cat{BR}}=\Gamma_{\BL_u}\cup\{(\to^{4}_{13}, \wedge^{4}_{14})\}.
\]
An application of Theorem~\ref{thm:GeneralProductEquiv} proves that 
$\cat{BR}$ is categorically equivalent to the variety of Brouwerian bilattices. 
Moreover, if we consider Heyting algebras (bounded Brouwerian lattices) and the 
duplicator
\[
\Gamma'_{\cat{H}}
=\Gamma_{\BL}\cup\{(\to^{4}_{13}, \wedge^{4}_{14})\}
\]
we can easily see that Heyting algebras are categorically equivalent to bounded Brouwerian bilattices. 
This leads to a categorical equivalence between bounded Brouwerian bilattices and $\krB$ that 
 is actually a term-equivalence.

\subsection*{Bilattices with truth implication}\

Here we  consider the class $\trB$  of bounded bilattices 
for which  $\land_t$ admits an adjoint.
  More precisely, an algebra  
  $\A=(A;\lor_t,\land_t,\lor_k,\land_k,\to_t,\neg, 0_t,1_t,0_k,1_k)$ belongs to 
$\trB$ if $(A;\lor_t,\land_t,\lor _k,\land_k,\neg, 0_t,1_t,0_k,1_k)$ is a bilattice and 
$(\land_t,\to_t)$ is an adjoint pair. 
Let $b\cat{H}$ be the class of bi-Heyting algebras (see \cite{RZ} and the references therein). 
We shall 
prove that the 
$\trB$ is a duplicate of $b\cat{H}$.

 We let $\A=(A;\lor_t,\land_t,\lor_k,\land_k,\to_k,\neg, 0_t,1_t,0_k,1_k)\in\trB$, and 
identify  $\A_{\BL}$ with  identify $\A_{\BL}$ with $\Lalg\odot\Lalg$ for some 
$\Lalg=(L;\lor,\land,0,1)\in\CL$. Since $(\land_t,\to_t)$ is an adjoint pair, we have, for 
$a,b,c\in L$, 
\begin{align*}
a\land b\leq  c &\Longleftrightarrow  (a,1)\land_t (b,1)\leq_t (c,1
)\\
&
 \Longleftrightarrow  (b,1)\leq_t (a,1)\to_t (c,1)
\\
&
  \Longleftrightarrow  b\leq \pi_{1}((a,1)\to_t (c,1)) 
\shortintertext{and}
a\lor b\geq c &\Longleftrightarrow (0,a)\wedge_t (0,b)\leq_t (0,c)
\\
&
 \Longleftrightarrow (0,b)\leq_t (0,a)\to_t (0,c)\\
& 
\Longleftrightarrow b\geq \pi_{2}((0,a)\to_t (0,c)). 
\end{align*}
 Thus  the binary operations $\to^{\Lalg}$ and $\mapsto^{\Lalg}$ 
 defined by 
 $x\to^{\Lalg} y= \pi_{1}((x,1)\to_t (y,1))$ and $x\mapsto^{\Lalg} y= \pi_{2}((0,x)\to_t (0,y))$ 
 are the adjoints of $\land$ and $\lor$, respectively. 
Hence  the algebra $(L;\lor,\land,\to^{\Lalg},\mapsto^{\Lalg},0,1)$  
belongs to $b\cat{H}$. 
Moreover, the set
\[
\Gamma_{b\cat{H}}=\Gamma_{\BL}\cup\{(\to^{4}_{13}, \mapsto^{4}_{24})\} 
\]
duplicates $b\cat{H}$.
Hence an application of Theorem~\ref{thm:GeneralProductEquiv}
proves our claim that $\trB$ is categorically equivalent to~$b\cat H$.

Combining the  ideas  of this section, we observe that if a bilattice 
is such that $\wedge_t$ has an adjoint, $\to_t$, 
then $\wedge_k$ also admits an adjoint. 
Moreover,  this adjoint can be captured  as follows:   
\[
x\to_k y=((x\to_t y)\wedge_k 1_t)\vee_k (\neg( \neg x\to_t\neg y)\wedge_k 0_t).
\]

An  analysis of a third scenario in which an implication is introduced into bilattices
is performed   in Example~\ref{ex:implic}, where we 
consider implicative bilattices, as these are defined in~\cite{AA1},  
and show how  they fit into a general scheme of Boolean algebra duplicates.


\section{Further examples}\label{Sec:Examples}

This section brings a non-exhaustive selection of examples within the scope of 
the general framework for product representations set up in Section~\ref{Sec:ProdRep}.
The examples concern the adjunction of new operations of different types to
different base varieties, and the identification of appropriate duplicates of these varieties.  
We group the examples according to the variety being duplicated.
 Thanks to Theorem~\ref{thm:GeneralProductEquiv}, the varieties within any such 
group are all categorically equivalent to one another, a fact which in many cases
has not been recognised before.

\subsection*{Lattice variety  duplicates} \ 

We have already mentioned that $\BL$, $\BLU$, $\DB$ and $\DBU$ 
are duplicates of $\CL$, $\CLU$, $\CCD$, and $\CCDU$, respectively. 
We now turn to new examples.

\begin{ex}\label{Ex:Guard} [Fitting's guard operation] 
Fitting  \cite{FKlkids} introduced a binary operation on $\fourvar{\DB}$,  denoted 
$\fitimp$ and   given by 
\[
a\fitimp b
=\begin{cases} 
b  & \text{if }a\in\{(1,1),(1.0)\} , \\
(0,0) &\text{otherwise}.
\end{cases}
\]
Observe that
 $(a_1,a_2)\fitimp (b_1,b_2)= ((a_1\wedge b_1),(a_1\wedge b_2))$. 
Let $\four_{\fitimp}$ be the algebra obtained by adding the operation
``\,$\fitimp$\,'' to $\fourvar{\DB}$.
It is easily seen that 
$\Gamma_{\DB}\cup\{(\wedge^4_{13},\wedge^{4}_{14})\}$ 
is a duplicator for $\Sigmavar {\CCD}$  on $\twovar{\CCD}$. 
By Theorem~\ref{thm:GeneralProductEquiv}, 
$\mathbb{V}(\four_{\fitimp})$ is categorically equivalent to 
$\CCD$.
\end{ex}

As we observed after Theorem~\ref{thm:GeneralProductEquiv} the
equivalence between a variety of algebras and its duplicate 
determines an isomorphism between the associated  lattices of subvarieties.
Moreover, we have observed that a duplicator for a variety is also a duplicator for any of its subvarieties. 
Now we will use this observation to get new base varieties and new duplicates from known duplicators.

We have already used a duplicator of De Morgan lattices to handle 
unbounded bilattices with conflation, and noted that a similar construction is available 
in the bounded case using De Morgan algebras. 
The variety $\DM$ has two  non-trivial proper 
subvarieties:  $\KL$ (Kleene algebras) and $\CB$ (Boolean algebras).
 The generators of the non-trivial  proper subvarieties of $\DM$  
also support various additional operations.   
We show how we  can obtain duplicators to capture such operations.  
These give rise to  product representations, old and new, 
 of algebras arising from the addition 
of various operations related to the De Morgan negation.


\subsection*{Kleene algebra duplicates}\label{ex:Kleeneduplicates} \

Let $\three_{\DM}=(\{0,u,1\};\vee,\wedge,\sim,0,1)$ 
denote the De Morgan algebra whose lattice reduct 
is  the three-element chain $0<u<1$.
The class $\ISP(\three_{\DM})$  is indeed 
a subvariety of $\DM$ (that is, $\ISP(\three_{\DM})=\HSP(\three_{\DM})$).
The algebras in $\ISP(\three_{\DM})$ are called Kleene algebras. Let $\KL$ denote the variety of Kleene algebras.
 The categorical equivalence between $\DM$ and $\DBC$ restricts 
 to a categorical equivalence between $\KL$ and 
$\ISP(\fnt{P}_{\Gammavari{\DBC}}(\threevar{\DM}))
=\mathbb{V}(\fnt{P}_{\Gammavari{\DBC}}(\threevar{\DM}))$.

\begin{ex}[Negation by failure]  
 \label{ex:failure}
In \cite{RF} Ruet and Faget introduce an operation called \emph{negation-by-failure} on the 
bilattice $\Ninevar{\DB}=\fnt{P}_{\Gammavari{\BL}}(\threevar{\CCD})$ 
(where $\threevar{\CCD}$ is the three-element lattice whose universe is $\{0,u,1\}$ and $0<u<1$)  
and the operator ${\slash\colon \Ninevar{\DB}\to \Ninevar{\DB}}$ is defined by 
\[  
\slash(a_1,a_2)=
\begin{cases}
(1-a_1,a_2) & \text{if } a_1=0 \text{ or } 1, \\
(a_1,a_2) &\text{otherwise}.
\end{cases}
\]
It follows that $\slash(a_1,a_2)=({\sim}a_1,a_2)$. 

Let $\NineS$ denote $\Ninevar{\DB}$ with the operation ``$\OurSlash$'' added.
It follows that 
$
\Gammavari{\slash}=\Gamma_{\DB}\cup\{(\sim^{2}_{1},x^{2}_2)\}
$ 
duplicates $\three_{\DM}$ and that
 $\NineS=\fnt{P}_{\Gammavari{\slash}}(\three_{\DM})$.
By Theorem~\ref{thm:GeneralProductEquiv}, 
$\HSP(\NineS)$ is equivalent to the variety of Kleene algebras.
\end{ex}

\subsection*{Boolean algebra duplicates} \ 

The class $\CB$ of Boolean algebras equals  $\ISP(\twovar{\CB})$ 
where $\twovar{\CB}=(\{0,1\};\vee,\wedge,'\, ,0,1)$
 is the two-element Boolean algebra.

\begin{ex}[Implicative bilattices]  \label{ex:implic}
In~\cite{AA1},  
Arieli and Avron  
considered a special implication operator definable on a logical bilattice 
(that is, a bilattice together with a prime bifilter). 
The case of $\fourvar{\DB}$ is very special, since  $\fourvar{\DB}$ 
only admits one bifilter, 
\textit{viz.}~$\{(1,1),(1,0)\}$. In this case the implication is given by 
\[
a\aaimp b
=\begin{cases}
b &\text{if } a\in\{(1,1),(1,0)\}, \\
(1,0) &\text{otherwise}.
\end{cases}
\]
In other words, 
$(a_1,a_2)\aaimp (b_1,b_2) = (a_1'\vee b_1,a_1\wedge b_2)$. 
Let 
\[
\four_{\aaimp}=(\{0,1\}^2;\lor_t,\land_t,\lor_k,\land_k,\aaimp,0_t,1_t,0_k,1_k) 
\]
be the algebra whose
bilattice reduct is $\fourvar{\DB}$ and $\aaimp$ is as defined above.
 Any algebra in the variety  $\mathbb{V}(\four_{\aaimp})$  is called an \emph{implicative bilattice}. Setting $t$ as the term 
$t(x_1,x_2,x_3,x_4)=x_1'\vee x_3$, it follows that the set 
$\Gammavari{\aaimp}=
\Gamma_{\BL}\cup\{(t ,\wedge^{4}_{14})\}$
duplicates $\twovar{\CB}$  and 
$\four_{\aaimp}=\fnt{P}_{\Gammavari{\aaimp}}(\twovar{\CB})$. 
By Theorem~\ref{thm:GeneralProductEquiv}, 
the variety $\mathbb{V}(\four_{\aaimp})$ of implicative bilattices
is categorically equivalent to $\CB$.

If we consider the unbounded reduct 
$\fourvar{\DB_u,\aaimp}=(\{0,1\}^2; \lor_t,\land_t,\lor_k,\land_k,\aaimp)$ 
of $\four_{\aaimp}$,
the set 
$\Gamma_{\scriptscriptstyle \BL_u}\cup\{(t ,\wedge^{4}_{14})\}$ 
duplicates $\twovar{\CGB}$, 
where $\CGB$ denotes the class of generalised (lower unbounded) Boolean algebras \cite{BD}, 
and hence $\mathbb{V}(\fourvar{\DBU,\aaimp})$ is equivalent to $\CGB$
 by Theorem~\ref{thm:GeneralProductEquiv}.
This equivalence was already observed in \cite{BJR11} as a 
consequence of the product representation of Brouwerian bilattices 
 and its application to implicative bilattices.
\end{ex} 

\begin{ex}[Moore's epistemic operator]
\label{Ex:Moore}
Ginsberg's interpretation of Moore's epistemic operator ``I know that $p$'' 
is the operation $L\colon \fourvar{\DB}\to \fourvar{\DB}$ defined by $L(a_1,a_2)=(a_1 ,a_1')$.

In \cite[Proposition 4.2]{GinMod} it is proved that the algebra 
\[
\four_{L}=(\{0,1\}^2;\vee_t,\wedge_t,\vee_k,\wedge_k,\neg, L)
\] 
is  primal. 
Therefore $\ISP(\four_{L})=\mathbb{V}(\four_{L})$.
We can obtain the same result independently from the primality of $\four_L$.
 Consider  the language $\Sigma_{\CB}$ of Boolean algebras.
 Trivially
\[
\Gammavari{L}=\Gammavari{\BL}\cup\{(x^2_1,(')^{2}_1)\}
\]
duplicates $\CB$. 
Moreover $\four_{L}=\fnt{P}_{\Gammavari{L}}(\twovar{\CB})$.
\end{ex}

\begin{ex}[Negation-by-failure on $\fourvar{\DB}$]
In \cite{RF}, Ruet and Faget consider their negation-by-failure operator restricted to 
$\fourvar{\DB}$, that is, $\slash\colon \fourvar{\DB}\to \fourvar{\DB}$ is defined by 
$\slash(a_1,a_2)=(1-a_1,a_2)$. 
Let $\four_{\slash}$ be the algebra obtained by enriching the language of 
$\fourvar{\DB}$ with $\slash$. It is easy to check that $\four_{\slash}$ is a subalgebra of $\NineS$. 
Moreover, by identifying $\twovar{\CB}$ with the two-element subalgebra  of $\threevar{\DM}$,  
it follows that $\four_{\slash}=\fnt{P}_{\Gammavari{\slash}}(\twovar{\CB})$, the set $\Gamma_{\slash}$ 
duplicates $\CB$, and the class 
$\ISP(\four_{\slash})=\HSP(\four_{\slash})=\HSP(\fnt{P}_{\Gammavari{\slash}}(\twovar{\CB}))$ 
is categorically equivalent to $\CB$.
\end{ex}

\subsection*{Duplicates of residuated lattices}\

 An algebra $\A=(A;\lor,\land,\, \cdot\,,\,\ldiv\, ,\,\rdiv\, )$ is said to be 
a \emph{residuated lattice}     if  $(A;\lor,\land)$ is a lattice
 and 
$a\cdot b\leq c  \Longleftrightarrow  b\leq a\ldiv c \Longleftrightarrow  a\leq c\rdiv b$
 (see for example \cite{GJKO}). Let us denote the variety of residuated lattices by $\cat{RL}$.

\begin{ex}[Residuated bilattices]  
\label{Ex:ResBil}
In \cite{JanR}, the authors defined the variety $\cat{RBL}$ of residuated bilattices. 
Using the notation of the present  paper and of \cite[Theorem~3.6]{JanR}
 it follows that $\cat{RBL}=\mathbb{V}(\fnt{P}_{\Gammavari{\cat{RBL}}}(\cat{RL)})$, where
${\Gammavari{\cat{RBL}}}
=\Gamma_{\BL}\cup
\{(\ldiv^{4}_{13},\cdot^{4}_{41}),(\rdiv^{4}_{13},\cdot^{4}_{32})\}$. 
Hence, Theorem~\ref{thm:GeneralProductEquiv} implies that $\cat{RBL}$ 
is categorically equivalent to $\cat{RL}$.
\end{ex}

\subsection*{Duplicates of modal algebras}\

Let $\cat{BM}$ be the variety of bi-modal algebras. 
An algebra $(A;\lor,\land,\, ' \, ,\square_{+},\square_{-}, 0,1)\in \cat{BM}$
if and only if $(A;\lor,\land,\, ' \, ,0,1)$ is a Boolean algebra and  
$\square_{+},\square_{-}\colon A\to A$ preserve finite meets.

\begin{ex}[Modal bilattices]  
\label{Ex:ModalBil}In \cite{JR13}, the authors studied a modal expansion of implicative bilattices. 
They presented a product representation for implicative bilattices with a modal operator. 
 An algebra $\A=(A;\lor_t,\land_t,\lor_k,\land_k,\aaimp,\neg,\square ,0_t,1_t,0_k,
1_k)$
is said to be 
a \emph{ modal bilattice}  if  $(A;\lor_t,\land_t,\lor_k,\land_k,\aaimp,\neg,0_t,1_t,0_k,
1_k)$  
 is an implicative lattice (see Example~\ref{ex:implic}) and 
\[
\square(1_t)=1_t,\quad \square(a\land_t b)=\square(a)\land_t\square(b),  
\quad 
\square(0_k\aaimp a )=0_k\aaimp \square(a).
\]
We denote the variety of modal bilattices by $\cat{MBL}$.

 It is easy to see that the set
 $ {\Gammavari{\cat{MBL}}}
=\Gammavari{\aaimp}\cup
\{(t_1,t_2)\}$, 
where $t_1(x_1,x_2)=\square_+(x_1)\wedge \square_{-}(x_2')$ and 
$t_2(x_1,x_2)=( \square_{+}(x_2'))'$, duplicates $\cat{BM}$. 
 The result of  \cite[Theorem~12]{JR13} proves that 
 $\cat{MBL}=\mathbb{V}(\fnt{P}_{\Gammavari{\cat{MBL}}}(\cat{BM}))$.
Hence, Theorem~\ref{thm:GeneralProductEquiv} implies that $\cat{BM}$ 
is categorically equivalent to $\cat{MBL}$.
\end{ex}

\section{Beyond 
product representation via duplication}
\label{Sec:Conclusion}

Our aim in writing this paper, as its title suggests, is to present 
 a general framework for product representations of classes of algebras. 
One may ask if Theorem~\ref{thm:GeneralProductEquiv} is the most 
general product representation we can obtain.  It is not.
In this section we indicate how
 our theorem can be extended 
in two different directions  (and in both simultaneously). 
Firstly we consider an extension to handle products which are not binary
and secondly we show how our duplication mechanism can 
be modified so that our methodology  encompasses product
representations which fall outside the scope of  duplication,  as  this appears in
 Theorem~\ref{thm:GeneralProductEquiv}. 
Our two variants will be put forward using 
a similar expository method in each case:
we first present a  pathfinder   example; 
then  we provide a modified version of
conditions (L), (M) and (P) to encompass this example; 
finally, we state the adaptation  
of Theorem~\ref{thm:GeneralProductEquiv} 
associated with the amended conditions.

Let us  consider our first modification of the product representation theorem. 
Our path\-finder example here is a new product representation for distributive trilattices. 
We have already observed that $\CCDU\dupl\DBU$ and $\DBU\dupl\TL_{t,f}$, and 
this  proves that $\TL_{t,f}$ is categorically equivalent to $\CCDU$.
This equivalence is determined by the composition of the functors
 $\fnt{P}_{\Gammavari{\DB_u}}$ and $\fnt{P}_{\Gammavari{\TL_{t,f}}}$.
Applying these two functors to  a distributive lattice $\Lalg$ would yield a 
trilattice whose universe is $L^4$ and whose operations are defined as follows:
 \begin{align*} 
(a_1,a_2,a_3,a_4) \lor_t (b_1,b_2,b_3,b_4)  & = (a_1 \lor b_1, a_2 \land b_2,a_3 \lor b_3, a_4 \land b_4), \\
(a_1,a_2,a_3,a_4) \land_t (b_1,b_2,b_3,b_4)  & = (a_1 \land b_1, a_2 \lor b_2,a_3 \land b_3, a_4 \lor b_4), \\
(a_1,a_2,a_3,a_4) \lor_f (b_1,b_2,b_3,b_4)  & = (a_1 \lor b_1, a_2 \lor b_2,a_3 \land b_3, a_4 \land b_4), \\
(a_1,a_2,a_3,a_4) \land_f (b_1,b_2,b_3,b_4)  & = (a_1 \land b_1, a_2 \land b_2,a_3 \lor b_3, a_4 \lor b_4), \\
(a_1,a_2,a_3,a_4) \lor_i   (b_1,b_2,b_3,b_4)  & = (a_1 \lor b_1, a_2 \lor b_2,a_3 \lor b_3, a_4 \lor b_4), \\
(a_1,a_2,a_3,a_4) \land_i (b_1,b_2,b_3,b_4)  & = (a_1 \land b_1, a_2 \land b_2,a_3 \land b_3, a_4 \land b_4), \\
-_{t}(a_1,a_2,a_3,a_4) &= (a_2,a_1,a_4,a_3),\\
-_{f}(a_1,a_2,a_3,a_4) &= (a_3,a_4,a_1,a_2).
\end{align*}

We shall now describe how to adapt (L),  (M) and (P) to yield a multi-factor 
product representation and thereby to obtain
$\TL_{t,f}$  directly from $\CCDU$ without going via $\DBU$. 
Again fix a class $\CN$  of $\Sigma$-algebras.
 But now let $\Gamma$ be a set of $m$-tuples of terms such that, 
 for each $\mathbf{t}=(t_1,\ldots,t_m)\in\Gamma$, 
there exists $n_{\mathbf{t}}\in\{0,1, \ldots\}$ such that $t_1,\ldots,t_m$ 
are terms on $m n_{\mathbf{t}}$ variables.  
We define 
\[
\fnt{P}^m_{\Gamma}(\N)
=(N^m; \{\mathbf{t}^{\fnt{P}^m_{\Gamma}(\N)}
\mid \mathbf{t}\in\Gamma \}),
\]
where   the operation
$\mathbf{t}\,^{\fnt{P}^m_{\Gamma}(\N)} \colon (N^m)^{n_{\mathbf{t}}}\to N^m$
 is  defined by
\[
\mathbf{t}\,^{\fnt{P}^m_{\Gamma}(\N)}(\mathbf{a}_1,\ldots, \mathbf{a}_{n_{\mathbf{t}}})
=
(t_1^{\N}(\mathbf{a}_1,\ldots, \mathbf{a}_{n_{\mathbf{t}}}),\ldots, t_m^{\N}(\mathbf{a}_1,\ldots, \mathbf{a}_{n_{\mathbf{t}}})), 
\quad
\text{for }\mathbf{a}_1,\ldots, \mathbf{a}_{n_{\mathbf{t}}}\in N^m.
\]
We extend our earlier notation  in the expected way:
given  a set $X$ we let $\delta_m^{X} \colon X\to X^{m}$ be the diagonal map 
given by  
$\delta_{m}^{X}(x)=(x,x,\ldots,x)\in X^m$ 
and, for $i\in\{1,\ldots, m\}$, 
let $\pi_i\colon X^m \to X$
 be the projection map onto the $i$th coordinate.

We consider the following generalisation of conditions (L), (M) and (P):
\begin{enumerate}  
\item[(L$_m$)] for each $n$-ary operation symbol $f\in\Sigma$ and  $i\in\{1,\ldots,m\}$ 
there exists an $n$-ary $\Gamma$-term $t$  such that 
$\pi^{N}_i\circ t^{\fnt{P}^m_{\Gamma}(\N)}\circ(\delta_{m}^{N})^{n}=f^{\N}$ 
for each $\N\in\CN$;  
\item[(M$_m$)] there exists an $m$-ary $\Gamma$-term $v$ such that 
\begin{multline*}
v^{\fnt{P}^{m}_{\Gamma}(\N)}((a^1_1,\ldots,a^1_m),\ldots,(a^m_1,\ldots,a^m_m))
=(a^1_1,a^2_2,\ldots,a^m_m)\\
\mbox{for $\N\in\CN$ and $(a^1_1,\ldots,a^1_m),\ldots,(a^m_1,\ldots,a^m_m)\in N^{m}$.}
\end{multline*}
\item[(P$_m$)] for each permutation $\sigma$ of $\{1,\ldots,m\}$ 
there exists a unary $\Gamma$-term $s_{\sigma}$ such that 
\[
s_{\sigma}^{\fnt{P}^{m}_{\Gamma}(\N)}(a_1,\ldots,a_n)
=(a_{\sigma(1)},a_{\sigma(2)},\ldots,a_{\sigma(n)})
 \quad \text{  for  $\N\in\CN$ and $a_1,\ldots,a_m\in N$.}
\]
\end{enumerate}

Observe that,  when $m=1$, the set $\Gamma$ 
consists of $\Sigma$-terms and conditions 
(M$_1$) and (P$_1$) are trivially satisfied. 
Moreover, condition (L$_1$) implies that
  $\mathbb{V}(\fnt{P}_{\Gamma}^1(\CN))$ is term-equivalent to $\mathbb{V}(\CN)$.
This justifies our observation that product representation is a generalised form of term-equivalence.

When $m=2$, conditions (L$_m$), (M$_m$) and (P$_m$) coincide with (L), (M) and (P). 
Thus Theorem~\ref{thm:GeneralProductEquiv} is a specialisation 
 of the following theorem, 
 whose proof follows using the same arguments and replacing  (L), (M) and (P) with (L$_m$), (M$_m$) and (P$_m$) 
as appropriate.

\begin{thm}\label{thm:mPR}
Let $\CN$ be a class of $\Sigma$-algebras  and  $\Gamma$  a set of $m$-tuples of 
$\Sigma$-terms.
 If $\Gamma$ satisfies {\upshape (L$_m$), (M$_{m}$)} and {\upshape (P$_m$)},
then the functor $\fnt{P}^{m}_{\Gamma}\colon \baseV\to \duplicateV$
sets up a categorical equivalence between 
$\baseV=\mathbb{V}(\CN)$ and 
 $\duplicateV=\mathbb{V}(\fnt{P}^{m}_{\Gamma}(\CN))$.
\end{thm}

\begin{ex}
It is easy to see that $\Gammavari{\TL_{t,f,i}}$ given by
\begin{multline*}
\Gammavari{\TL_{t,f,i}}=
\bigl\{
(\land^8_{15},\lor^8_{26},\land^8_{37},\lor^8_{48}),
(\lor^8_{15},\land^8_{26},\lor^8_{37},\land^8_{48}),
(\land^8_{15},\land^8_{26},\lor^8_{37},\lor^8_{48}),\\
(\lor^8_{15},\lor^8_{26},\land^8_{37},\land^8_{48}),
(\land^8_{15},\land^8_{26},\land^8_{37},\land^8_{48}),
(\lor^8_{15},\lor^8_{26},\lor^8_{37},\lor^8_{48}),\\
(x^4_{2},x^{4}_1,x^{4}_4,x^4_3),
(\sim^4_{2},\sim^{4}_1,\sim^{4}_4,\sim^{4}_3),
(x^4_{3},x^{4}_4,x^{4}_1,x^4_2)
\bigr\}
\end{multline*}
satisfies (L$_4$), (P$_4$) and  (M$_4$) 
with respect to
$\DMU$. 
Moreover $\mathbf{256}\cong \fnt{P}^4_{\Gammavari{\TL_{t,f,i}}}(\fourvar {\DM_u} )$. 
Combining Theorem~\ref{thm:mPR} and Lemma~\ref{lem:ISPTrilattices}(iv), 
it follows that $\DMU$ is categorically equivalent to $\TL_{t,f,i}$. 
The same result can be obtained from $\DMU\dupl\DBCU$ and $\DBCU\dupl \TL_{t,f,i}$ 
and two applications of Theorem~\ref{thm:GeneralProductEquiv}.
\end{ex}

Our presentation of our second variant of product representation starts from
consideration of the class of interlaced pre-bilattices.
An algebra $\A=(A;\lor_{t},\land_{t},\lor_{k},\land_{k})$ 
is a \emph{pre-bilattice} 
if both reducts $(A;\lor_{t},\land_{t})$ and $(A;\lor_{k},\land_{k})$ 
are lattices.
Pre-bilattices form a variety, $p\BLU$;  in fact  $p\BLU$ 
is the variety generated by the $\neg$-free   
reducts of (unbounded) bilattices.
A pre-bilattice is \emph{interlaced}  if each lattice operation is 
 monotonic with respect to the order of the  other lattice. 
There is a product representation for  
pre-bilattices (see \cite{D13} and the references therein).  
It follows the same lines as  that for bilattices, 
except that, in the absence of~$\neg$, 
 the two factors do not have to have the same universe
and the two coordinates operate independently.
We now formulate this precisely.
Let $\PP,\Q\in\CLU$.  Then $\PP\odot\Q$ is the pre-bilattice whose universe is 
$P\times Q$ and whose operations are defined by:
 \begin{alignat*}{2}
(a_1,a_2) \lor_t (b_1,b_2)  & = (a_1 \lor b_1, a_2 \land b_2), 
\qquad & 
(a_1,a_2) \lor_k (b_1,b_2)  & = (a_1 \lor b_1, a_2 \lor b_2), \\
(a_1,a_2) \land_t (b_1,b_2)  & = (a_1 \land b_1, a_2 \lor b_2),
 \qquad &
(a_1,a_2) \land _k (b_1,b_2)  & = (a_1 \land b_1, a_2 \land b_2).
\end{alignat*}
Pre-bilattices of the form $\PP \odot \Q$ are necessarily interlaced.
The product representation theorem for pre-bilattices  states that each 
interlaced pre-bilattice $\A$ is isomorphic to $\PP\odot \Q$ for some  $\PP,\Q\in\CLU$.
 Moreover this product representation can be upgraded  to a categorical equivalence 
 between $\CLU\times\CLU$ and the variety of interlaced pre-bilattices \cite[Section~5.1]{BJR11}.
 
Our next step is to modify the conditions (L), (M) and (P) to be imposed on a set
 $\Gamma$ so as to encompass the  example of pre-bilattices.
Condition (P), on permutation of coordinates, serves  to link the factors in a product. 
We want to dispense with this and to replace by  it by  a condition, (D), which
distinguishes coordinates in such a way that the factors in a product operate independently.    
We now indicate how this should work.

Let us fix a class $\CN$  of $\Sigma$-algebras and let $\Gamma$ 
be a set of pairs of $\Sigma$-terms.
Presented with two algebras $\PP,\Q\in\CN$ we want to use~$\Gamma$ 
to obtain an algebra $\PP\odotG{\Gamma}\Q$ whose universe is $P\times Q$.
Certainly condition (P) cannot be satisfied and 
the pairs of terms $(t_1,t_2)\in\Gamma$ should not combine 
elements from different coordinates.
More precisely, in order  for the operation
$[t_1,t_2]^{\PP\odotG{\Gamma}\Q} \colon (P\times Q)^{n}\to P\times Q$, 
given by
\begin{multline*}
[t_1,t_2]^{\PP\odotG{\Gamma}\Q}
((a_1,b_1),\ldots,(a_{n},b_{n}))=
(t_1^{\PP}(a_1,b_1,\ldots,a_{n},b_{n}),t_2^{\Q}(a_1,b_1,\ldots,a_{n},b_{n})),
\\
\text{for }(a_1,b_1),\ldots,(a_{n},b_{n})\in P\times Q,
\end{multline*}
where $n=n_{(t_1,t_2)}$,
to be well defined, we need  
$\Gamma$ 
to satisfy a condition that keeps the use of  the coordinates disjoint:
\begin{newlist}
\item[(D)] for each $(t_1,t_2)\in\Gamma$, 
\[
 t_1(x_1,\ldots,x_{2n})=r_1(x_{1},x_3,\ldots, x_{2n-1})\mbox{ and  }t_2(x_1,\ldots,x_{2n})=r_2(x_{2},x_4,\ldots, x_{2n}),
 \]
 for some $n$-ary $\Sigma$-terms $r_1$ and $r_2$.
\end{newlist}
Indeed, if $\Gamma$ satisfies (D) is easy to see that the algebra 
\[
\PP\odotG{\Gamma}\Q
=( P\times Q; \{ [t_1,t_2]^{\PP\odotG{\Gamma}\Q}\mid (t_1,t_2)\in\Gamma \} )
\]
is well defined whenever $\PP,\Q\in \mathbb{V}(\CN)$. 
Moreover, the functor
$\odotG{\Gamma}\colon \baseV\times \baseV\to\duplicateV$, 
where $\baseV=\mathbb{V}(\CN)$  and $\duplicateV=\mathbb{V}(\{\PP\odotG{\Gamma}\Q\mid \PP,\Q\in\CN\})$,
given by
\begin{alignat*}{3}
&\text{on objects:} & \hspace*{2.5cm}  & (\PP,\Q) \mapsto \PP\odotG{\Gamma}\Q , 
  \hspace*{2.5cm}  \phantom{\text{on objects:}}&&\\
&\text{on morphisms:}  & &  \odotG{\Gamma}(h_1,h_2) (a,b)=(h_1(a),h_2(b)).
\end{alignat*}
 is also well defined.

We now have a candidate set of conditions 
for a new product decomposition theorem.
Its proof is a straightforward modification of 
that of  Theorem~\ref{thm:GeneralProductEquiv}. 

\begin{thm}\label{thm:ICP}
Let $\CN$  be a class of $\Sigma$-algebras and  $\Gamma$ 
be a set of pairs of $\Sigma$-terms. Assume that  $\Gamma$ satisfies 
{\upshape (L), (M)} and {\upshape(D)}.  
Then the functor $\odotG{\Gamma}\colon \baseV\times \baseV\to\duplicateV$,
sets up a categorical equivalence between 
$\baseV\times \baseV$ {\upshape(}where $\baseV=\mathbb{V}(\CN)${\upshape)}  
and $\duplicateV=\mathbb{V}(\{\PP\odotG{\Gamma}\Q\mid \PP,\Q\in\CN\})$.
\end{thm}

Corollary~\ref{Cor:IrreducibleProjective} gives easy access to algebraic facts about 
varieties to which Theorem~\ref{thm:GeneralProductEquiv} applies. 
 A corresponding corollary to Theorem~\ref{thm:ICP} can be formulated.

\begin{table}  [b]
\begin{tabular}{|l|l |l|}
\hline
&&\\[-3.5mm]
variety& duplicate of& reference\\
\hline&&\\[-3.5mm]
bilattices & lattices & \multirow{4}{*}{Section~\ref{Sec:BilatticeProd}}\\
 $\BL$ ($\BLU$) & $\CL$ ($\CLU$)  &\\
\cline{1-2}&&\\[-3.5mm]
distributive bilattices & distributive lattices & \\
$\DB$ ($\DBU$) & $\CCD$ ($\CCDU$) &\\
\hline&&\\[-3.5mm]
distributive bilattices with conflation & De Morgan algebras (lattices) &\multirow{2}{*}{Section~\ref{sec:Conflation}}\\
$\DBC$ ($\DBCU$) &
$\DM$ ($\DMU$)&\\
\hline&&\\[-3.5mm]
distributive trilattices   & 
\hspace{-2mm}\multirow{3}{*}{\begin{tabular}{l} distributive bilattices \\
$\DBU$
\end{tabular}}
& \multirow{6}{*}{Section~\ref{Sec:Trilattices}}\\
with $t$- and $f$-involution &&\\
$\TL_{t,f}$& &\\
&&\\[-3.5mm]
\cline{1-2}&&\\[-3.5mm]
distributive trilattices  & distributive bilattices &\\
with $t$-, $f$- and $i$-involution &with conflation&\\
$\TL_{t,f,i}$& $\DBCU$&\\
&&\\[-3.5mm]
\hline&&\\[-3.5mm]
bilattices with knowledge implication & Heyting algebras &  \multirow{4}{*}{Section~\ref{Sec:Residuated}}\\
$\BL_{\to_{k}}$& $\cat{H}$&\\
&&\\[-3.5mm]
\cline{1-2}&&\\[-3.5mm]
bilattices with truth implication & bi-Heyting algebras &\\
$\BL_{\to_{t}}$& $b\cat{H}$&\\
&&\\[-3.5mm]
\hline&&\\[-3.5mm]
bilattices with guard operator& distributive lattices &  \multirow{2}{*}{Example~\ref{Ex:Guard}}\\
 $\mathbb{V}(\four_{\fitimp})$  & $\CCD$ &\\
 &&\\[-3.5mm]
 \hline&&\\[-3.5mm]
bilattices with negation by failure & Kleene algebras &  \multirow{2}{*}{Example~\ref{ex:failure} }\\
$\mathbb{V}(\NineS)$  & $\cat{KL}$ &\\
&&\\[-3.5mm]
\hline&&\\[-3.5mm]
 implicative bilattices & Boolean algebras &  \multirow{4}{*}{Example~\ref{ex:implic}}\\
 $\mathbb{V}(\four_{\aaimp})$ & $\CB$ &\\
 &&\\[-3.5mm]
 \cline{1-2}&&\\[-3.5mm]
unbounded implicative bilattices & generalised Boolean algebras &  \\
 $\mathbb{V}(\fourvar{\DB_u,\aaimp})$ & $\CGB$ &\\
 &&\\[-3.5mm]
 \hline &&\\[-3.5mm]
bilattices with Moore's  & 
\hspace{-2mm}\multirow{3}{*}{\begin{tabular}{l} 
Boolean algebras\\
 $\CB$ 
\end{tabular}}
 &\hspace{-2mm} \multirow{3}{*}{ Example~\ref{Ex:Moore} }\\
epistemic operator&&\\
$\mathbb{V}(\four_{L})$ && \\
&&\\[-3.5mm]
\hline&&\\[-3.5mm]
residuated bilattices & residuated lattices &  \multirow{2}{*}{Example~\ref{Ex:ResBil}}\\
$\cat{RL}$ & $\cat{RBL}$&\\
\hline&&\\[-3.5mm]
modal bilattices & bi-modal algebras &  \multirow{2}{*}{Example~\ref{Ex:ModalBil}} \\
$\cat{MBL}$ & $\cat{BM}$&\\[.2ex]
\hline
\end{tabular}
\\[2ex]
\caption{Varieties obtained by duplication\label{table1}}
\end{table}

\begin{ex}[Interlaced trilattices]\label{Ex:IntTri}
In \cite{R13}, Rivieccio presented product representations for the varieties of interlaced trilattices and interlaced trilattices with one involution $-_{t}$. 
A trilattice is said to be interlaced if the six lattice operations preserve each of the three orders.  

For 
(unbounded)
interlaced trilattices, $\IT$, 
we take the base variety to be $p\BLU$,
the variety of pre-bilattices, 
 and define 
\begin{multline*}
\Gammavari{\IT}=\{
((\land_t)^4_{13},(\land_t)^4_{24}),
((\lor_t)^4_{13},(\lor_t)^4_{24}),
((\lor_k)^4_{13},(\land_k)^4_{24}),
((\land_k)^4_{13},(\lor_k)^4_{24}),\\
((\land_k)^4_{13},(\land_k)^4_{24}),
((\lor_k)^4_{13},(\lor_k)^4_{24})
\}.
\end{multline*}
Then the product representation theorem for $\IT$ \cite[Theorem~3.4]{R13} can be formulated as the assertion 
 that $\IT=\mathbb{V}(p\BLU\odotG{\Gammavari{\IT}}p\BLU)$.
 Moreover, since $\Gammavari{\IT}$ certainly satisfies 
(L), (M) and (D),  Theorem~\ref{thm:ICP} implies that 
$\IT$ is categorically equivalent to $p\BLU\times p\BLU$.

Now consider 
the variety $\IT_{-_{t}}$ of interlaced trilattices with $t$-involution. 
Let $\BL_u$ be the base variety and let $\Gammavari{\IT}$  
be the following set of pairs of terms in the language of $\BL_u$:
\[
\Gammavari{\IT_{-_{t}}}=
\Gammavari{\IT}\cup\{(\neg^{2}_{1},\neg^{2}_2)\}.
\]
Then \cite[Theorem~3.6]{R13} proves that  $\IT_{-_{t}}=\mathbb{V}(\BL\odotG{\Gammavari{\IT_{-_{t}}}}\BL_u)$. By Theorem~\ref{thm:ICP}, it follows that 
$\IT_{-_{t}}$ is categorically equivalent to $\BL_u\times\BL_u$.
\end{ex}

Of course we could combine the  generalisation   to $m$-factor products  and 
the variant that allows different components in the resulting product.
 Specifically we could introduce a  condition (D$_m$) and,
 by applying  to (L$_m$), (M$_{m}$) and (P$_m$) 
the same reasoning that we used to replace (M) by (D) in Theorem~\ref{thm:ICP},
 obtain  a categorical equivalence between 
$(\mathbb{V}(\CN))^m$ and $\mathbb{V}(\{\N_1\odotG{\Gamma}\cdots\odotG{\Gamma}\N_m\mid \N_1,\ldots,\N_m\in\CN\})$.  We omit the details.
By this means we can in particular  arrive at  a direct proof that $\TL$
is categorically equivalent to 
$\CCDU \times \CCDU \times \CCDU \times \CCDU$
or that $\IT$ is categorically equivalent to $\CLU\times \CLU\times\CLU\times\CLU$.

\section*{Appendix:  summary of duplications and equivalences}

For reference, and to emphasise the uniformity of our approach to product representations across a wide range of varieties we include  two tables 
summarising our results.       

The first table covers varieties to which conditions (L), (P) and (M) of 
Section~\ref{Sec:ProdRep} apply.  Any two varieties in the same row
are categorically equivalent, and any two duplicates with a common 
base variety are equivalent to each other.  
This table may be seen as an amplified version of that given by Jung and Rivieccio \cite{JRtacl}.  We stress that we  are able to view all the examples in our table as being underpinned 
by a common syntactic mechanism.

Table~\ref{table2} serves a somewhat different purpose from Table~\ref{table1}.
It compares and contrasts the behaviour of (interlaced) trilattices with 
different numbers of involutions added, from none to three.   
We have already seen in Section~\ref{Sec:Trilattices} how Theorem~\ref{thm:GeneralProductEquiv} can be
employed to obtain categorical equivalences.  Here we focus on the use of 
the ideas in Section~\ref{Sec:Conclusion}. 

\begin{table} [ht]
\begin{tabular}{|l|l|l|}

\hline&&\\[-3.5mm]
variety& equivalent to&reference\\
\hline&&\\[-3.5mm]
distributive trilattices  
 &
 \hspace{-2mm}\multirow{3}{*}{\begin{tabular}{l} distributive lattices \\$\CCDU$ \end{tabular}}&\\
with $t$- and $f$-involution &&Theorem~\ref{thm:mPR}\\
$\TL_{t,f}$&& \\
&&\\[-3.5mm]
\hline&&\\[-3.5mm]
distributive trilattices   & 
\hspace{-2mm}\multirow{3}{*}{\begin{tabular}{l} De Morgan lattices \\
 $\DMU$
\end{tabular}}
&\\
with $t$- $f$- and $i$-involutions&&Theorem~\ref{thm:mPR}\\
$\TL_{t,f,i}$&&\\
&&\\[-3.5mm]
\hline&&\\[-3.5mm]
pre-bilattices & lattices $\times$ lattices &\hspace{-2mm}\multirow{2}{*}{\begin{tabular}{l} Theorem~\ref{thm:ICP} \end{tabular}}\\
$p\BLU$& $\CLU\times\CLU$&\\
&&\\[-3.5mm]
\hline&&\\[-3.5mm]
 & pre-bilattices  $\times$ pre-bilattices & Theorem~\ref{thm:ICP}\\
\hspace{-2mm}
\multirow{3}{*}{\begin{tabular}{l} interlaced trilattices \\
$\IT$ \end{tabular}}
& $p\BLU\times p\BLU$& \\
&OR&\\
&lattices $\times$ lattices $\times$ lattices $\times$ lattices&Theorem ~\ref{thm:ICP}
\\
& $\CLU\times\CLU\times\CLU\times\CLU$&(4-factor version) \\
\hline&&\\[-3.5mm]
\multirow{3}{*}{\begin{tabular}{l} distributive  trilattices
 \\
$\TL$ \end{tabular}}
& $p\DBU\times p\DBU$&Theorem~\ref{thm:ICP} \\
&OR&
\\
& $\CCDU\times\CCDU\times\CCDU\times\CCDU$&Theorem~\ref{thm:ICP}\\
&& (4-factor version) \\
\hline&&\\[-3.5mm]
& bilattices $\times$ bilattices&Theorem~\ref{thm:ICP}\\
interlaced trilattices  &  $\BLU\times\BLU$&
\\
with $t$-involution & OR&\\
$\IT_{-_{t}}$& lattices $\times$ lattices&Theorem~\ref{thm:ICP}\\
& $\CLU\times\CLU$&(\& Theorem~\ref{thm:GeneralProductEquiv})\\
\hline&&\\[-3.5mm]
distributive  trilattices& $\DBU\times \DBU$& \\
with $t$-involution &OR&Theorem ~\ref{thm:ICP}
\\
$\TL_{-_{t}}$ & $\CCDU\times\CCDU$& \\[1mm]
\hline
\end{tabular}
\\[2ex]
\caption{Equivalences derived from 
Theorems~\ref{thm:mPR} and~\ref{thm:ICP} 
(no bounds)
 \label{table2}}
\end{table}

\section{Funding}
This work was supported by the [European Community's] Seventh Framework Programme [FP7/2007-2013] under the Grant Agreement n. 326202 to L.M.C.

\end{document}